\documentclass[11pt,english]{amsart}
\def\margin_comment#1{\marginpar{\sffamily{\tiny #1\par}\normalfont}}
\usepackage{graphicx}
\usepackage{setspace}
\usepackage{color}
\usepackage{amsmath,amssymb,amsthm}
\usepackage{epsfig}
\usepackage{tikz}
 \usepackage{amsmath}
\usepackage{tikz-cd}
\usetikzlibrary{matrix, calc, arrows}
\usepackage{tikz}
\usetikzlibrary{arrows,chains,matrix,positioning,scopes}
\usepackage{lipsum}

\newcommand{\Div}[1]{\operatorname{Div}(#1)}

\makeatletter
\tikzset{join/.code=\tikzset{after node path={%
\ifx\tikzchainprevious\pgfutil@empty\else(\tikzchainprevious)%
edge[every join]#1(\tikzchaincurrent)\fi}}}
\makeatother
\tikzset{>=stealth',every on chain/.append style={join},
         every join/.style={->}}
\tikzstyle{labeled}=[execute at begin node=$\scriptstyle,
   execute at end node=$]
\newtheorem{thm}{Theorem}[section]
\numberwithin{equation}{section} 
\numberwithin{figure}{section} 
\theoremstyle{plain}
\newtheorem*{thm*}{Theorem}
\theoremstyle{definition}
\theoremstyle{plain}
\newtheorem{thm_A}{Theorem}
\newtheorem*{defn*}{Definition}

\theoremstyle{plain}

\theoremstyle{plain} 

\theoremstyle{plain}
\newtheorem{prop}[thm]{Proposition} 
\theoremstyle{remark}
\newtheorem{ex}[thm]{Example}
\theoremstyle{remark}
\newtheorem{rem}[thm]{Remark}
\theoremstyle{plain}

\theoremstyle{plain}

\theoremstyle{plain}
\newtheorem{lem}[thm]{Lemma}
\newtheorem*{lem*}{Lemma} 
\theoremstyle{definition}
\newtheorem{defn}[thm]{Definition}
\newtheorem*{acknowledgment}{Acknowledgment}
\newtheorem*{acknowledgment*}{Addentum}

\theoremstyle{plain}
\newtheorem*{ex*}{Example}
\theoremstyle{plain}

\AtBeginDocument{
  
}
\begin{document}
\title[A group of automorphisms for a family of Garside groups]{Construction of a group of automorphisms for an infinite family of Garside groups}
\author{Fabienne Chouraqui}
\begin{abstract}
The structure groups of  non-degenerate  symmetric set-theoretical solutions of the quantum Yang-Baxter equation provide an infinite family of Garside groups with many interesting properties. Given a non-degenerate  symmetric solution, we construct for  its  structure group a group of automorphisms. Moreover, we show this group of automorphisms admits a subgroup that preserves the Garside properties of  the structure group. In some cases, we could also prove the group of automorphisms obtained is an outer automorphism group.
\end{abstract}
\maketitle
\section*{Introduction}
The quantum Yang-Baxter equation is an equation in the field of mathematical physics and it lies in the foundation of the theory of quantum groups.
Let $R:V  \otimes V \rightarrow V  \otimes V$ be a linear operator, where $V$ is a vector space. The quantum Yang-Baxter equation is the equality $R^{12}R^{13}R^{23}=R^{23}R^{13}R^{12}$ of linear transformations on $V  \otimes V \otimes V$, where $R^{ij}$ means $R$ acting on the $i$th and $j$th components.  
V. Drinfeld suggested the study of a set-theoretical solution, that is  a pair $(X,S)$ for which $V$ is the  vector space spanned by the set $X$ and $R$ is the linear operator induced by the mapping $S: X \times X \rightarrow X \times X$.  To each  non-degenerate and symmetric  set-theoretical solution $(X,S)$ of the quantum Yang-Baxter equation, P.Etingof, T.Schedler and A.Soloviev associate a group $G(X,S)$ called the structure group. Furthermore, they show 
that each non-degenerate symmetric solution $(X,S)$ is  in one-to-one correspondence with a quadruple $(G,X,\rho, \pi)$, where $G$ is a group, $X$ is a set, $\rho$ is a left action of $G$ on $X$, and $\pi$ is a bijective $1$-cocycle of $G$ with coefficients in $\mathbb{Z}^{X}$, where $\mathbb{Z}^{X}$ is the free abelian group generated by $X$ \cite{etingof}.

In this paper, given a non-degenerate symmetric solution $(X,S)$ with $\mid X \mid =n$ and structure group $G(X,S)$, we use the bijective $1$-cocycle $\pi:G(X,S) \rightarrow \mathbb{Z}^{X}$ and the fact that $\operatorname{Aut}(\mathbb{Z}^{X})=  GL_n (\mathbb{Z})$ to construct a group of automorphisms of $G(X,S)$. Indeed, given   $\sigma \in GL_n (\mathbb{Z})$, we define a bijection $\varphi: G(X,S) \rightarrow G(X,S)$ by 
$\varphi=\, \pi^{-1} \circ \sigma \circ \pi $, such that the following diagram is commutative:
\begin{center} \begin{tikzpicture}
   \matrix (m) [matrix of math nodes,row sep=3em,column sep=4em,minimum width=2em]
   {
      G(X,S) & \mathbb{Z}^{X} \\
      G(X,S) & \mathbb{Z}^{X} \\};
   \path[-stealth]
     (m-1-1) edge node [left] {$\varphi$} (m-2-1)
             edge [double] node [above] {$\pi$} (m-1-2)
     (m-2-1.east|-m-2-2) edge [double]
             node [above] {$\pi$} (m-2-2)
     (m-1-2) edge node [right] {$\sigma$} (m-2-2)
             ;
 \end{tikzpicture} \end{center}
 We find a necessary and sufficient condition on   $\sigma$  ensuring that $\varphi=\, \pi^{-1} \circ \sigma \circ \pi $ is a homomorphism of the group $G(X,S)$.  Note that depending on the context, we consider $\sigma \in \operatorname{Aut}(\mathbb{Z}^{X})=  GL_n (\mathbb{Z})$ as a map or as matrix and  we write the elements in $\mathbb{Z}^{X}$ multiplicatively. We prove the following result  and we refer to Theorem \ref{theo_im_group_auto}, where another equivalent condition is given and more generally to  Section \ref{sec_mainResult}.
 \begin{thm_A}\label{intro_theo_aut_subgp}
   Let $(X,S)$ be  a non-degenerate symmetric set-theoretical solution of the quantum Yang-Baxter equation, with $\mid X \mid =n$. Let $(G(X,S),X,\bullet, \pi)$ be  the  corresponding quadruple, where  $G(X,S)$ is the structure group, $\bullet$ denotes the left action  of  $G(X,S)$ on $X$ by permutations extended to $\mathbb{Z}^{X}$ and  $\pi: G(X,S) \rightarrow \mathbb{Z}^{X}$ is the  bijective $1$-cocycle defined by $\pi(x)=t_x$, $x \in X$ and  $\pi(a_1a_2)=(a_2^{-1}\bullet \pi(a_1))\,\pi(a_2)$, $a_1,a_2 \in G(X,S)$.  
  Let $\Im_\pi$  to be the following set: 
      \begin{eqnarray*}
    \Im_\pi=\,\{\sigma \in GL_n (\mathbb{Z}) \mid  A_j\,\sigma\,=\,\sigma\,(f_j^{-1}),\,  \forall 1 \leq j\leq n \}
     \end{eqnarray*}
   where $(f_j^{-1})$ and  $A_j$ are the permutation matrices corresponding to the action of $x_j$ and  $(\varphi(x_j))^{-1}$ respectively. We define  $\Phi_{\pi}$ to be the following set: 
                 \begin{equation}\Phi_{\pi} =\{ \varphi: G(X,S) \rightarrow G(X,S), \;\;\;\varphi= \pi^{-1} \circ \sigma \circ \pi\, \mid \sigma \in \Im_\pi \}  \end{equation} 
            Then $\Im_\pi \leq  GL_n (\mathbb{Z})$ and $\Phi_{\pi}\leq \operatorname{Aut}(G(X,S))$.   
    \end{thm_A}
    There exists a strong  connection between set-theoretical solutions and Garside groups \cite{chou_art,chou_godel}, \cite{gateva}, \cite{ddkgm}. Indeed,  there is a one-to-one correspondence between non-degenerate symmetric set-theoretical solutions of the quantum Yang-Baxter equation and Garside groups with a particular presentation. The structure group $G(X,S)$ of $(X,S)$ is a Garside group that admits a Garside element of length $n=\mid X \mid $ which is the lcm of $X$ with respect to left and right divisibility \cite{chou_art}. Furthermore, $G(X,S)$ admits  a normal free abelian subgroup  $N$ of rank $n$,  such that the finite quotient group $G(X,S)/N$ plays the role the finite Coxeter groups play for finite-type Artin groups \cite{chou_godel2}, \cite{deh_coxeterlike}. The subgroup $N$ is freely generated by the so-called frozen elements $\theta_1, \theta_2,..\theta_n$. Whenever   $\varphi$ is an automorphism of the group $G(X,S)$, $\varphi \in \Phi_{\pi}$, a question that arises naturally is  whether $\varphi$ preserves the Garside structure of $G(X,S)$. The answer is negative in general, as $\varphi \in \Phi_{\pi}$ is not necessarily an automorphism of the underlying monoid. Nevertheless, we show that $\Phi_{\pi}$ admits a subgroup of automorphisms that preserves the Garside structure of $G(X,S)$ and  $N$; the automorphisms in this subgroup are induced by permutation matrices in   $\Im_\pi$.
 \begin{thm_A}\label{intro_theo_garside}
 Let  $(X,S)$ be a non-degenerate symmetric  solution, with  structure group  $G(X,S)$ and $\mid X \mid =n$. Let 
 $\Delta$ be the Garside element of length $n$ and  $\operatorname{Div}(\Delta)$ be the set of left (and right) divisors of  $\Delta$. Let  $N$ be the free abelian normal subgroup of  $G(X,S)$ generated by the frozen elements $\theta_1, \theta_2,..\theta_n$. Let $\sigma$ be a  permutation matrix in  $\Im_\pi$ and let $\varphi \in \operatorname{Aut}(G(X,S))$ be induced  by $\sigma$.  Then 
 $\varphi$ is a length-preserving isomorphism of the bounded lattice $\operatorname{Div}(\Delta)$ w.r to right and left divisibility. Furthermore, $\varphi(N)=N$.
  \end{thm_A}
To some extent, we could extend the results of Theorem \ref{intro_theo_garside} to the subgroup of $\Phi_\pi$ induced by the generalized permutation matrices in $\Im_\pi$. We refer to Proposition \ref{prop_gene.permut_preserve_extended_garside}.  Indecomposable solutions are the ``building blocks'' of the solutions and there is a parallel between the role of the indecomposable solutions and the irreducible components of  a finite-type Artin group. But, while a finite-type Artin group is the direct product of its irreducible components, this is not necessarily the case for the structure group of a decomposable solution. For  indecomposable solutions, the groups   $\Im_\pi$ and $\Phi_\pi$ satisfy some properties that are very specific to them and in particular  for the (unique) indecomposable solutions with $\mid X \mid =2,3$,  we show the group $\Phi_{\pi}$ is an outer automorphism group. More precisely, we prove:
         \begin{thm_A}\label{intro_theo_out}
         Let  $(X,S)$ be an indecomposable non-degenerate symmetric  solution with structure group $G(X,S)$.  Assume $\mid X \mid=2,3$. Then\\
          $(i)$ $\Im_\pi$  is a finite group of generalized permutation matrices in $\operatorname{GL_n(\mathbb{Z})}$.\\
          $(ii)$ $\Phi_{\pi}$ has a normal subgroup of index $2$ that preserves the Garside structure.\\
           $(iii)$ $\Phi_{\pi} =\operatorname{Out}(G)$ for $\mid X \mid=2$ and $\Phi_{\pi} \leq \operatorname{Out}(G)$ for $\mid X \mid=3$.
         \end{thm_A}  
    The paper is organised as it follows. In Section~\ref{sec_QYBE}, we introduce the structure group of a set-theoretical solution of the quantum Yang-Baxter equation, the bijective  $1$-cocycle and its properties. In Section \ref{Sec_background_garside}, we recall some basic material on Garside theory and we present the Garside structure specific to the structure group of a set-theoretical solution of the quantum Yang-Baxter equation. In Section \ref{sec_mainResult}, we describe the construction of the  group of automorphisms $\Phi_\pi$, we prove Theorem \ref{intro_theo_aut_subgp} and  we show the invariance of $\Phi_\pi$ under some particular changes of bijective  $1$-cocycle.
 In Section \ref{sec_generalized}, we show the existence of a group of automorphisms that preserves the Garside structure of  the structure group and we prove Theorem \ref{intro_theo_garside}. 
In Section \ref{Sec_indecomposable}, we consider the special case of indecomposable solutions, we prove some properties  of $\Im_\pi$ and $\Phi_\pi$ that are specific to general indecomposable solutions and we prove Theorem \ref{intro_theo_out}. In Section \ref{sec_Application_to_quotientgroup}, we present an application of our construction to the Coxeter-like quotient group $G(X,S)/N$. At the end of the paper, we add an appendix, in which we give for completion the proofs of the properties of the bijective 1-cocycle that appear in Section 1.2.
\begin{acknowledgment}
I am very grateful to
 Arye Juhasz and Yuval Ginosar for fruitful discussions.
\end{acknowledgment}

\section{Set-theoretical solutions of the quantum Yang-Baxter equation}\label{sec_QYBE}

\subsection{Background on set-theoretical solutions of the QYBE}
\label{subsec_qybe_Backgd}
We follow~\cite{etingof} and refer to~\cite{etingof}, \cite{jespers_book} and \cite{ginosar,ginosar2,cedoetall,gateva_van,jespers} for more details.
Fix a finite dimensional vector space~$V$ over the field~$\mathbb{R}$. The Quantum Yang-Baxter Equation on~$V$ is the equality $R^{12}R^{13}R^{23}=R^{23}R^{13}R^{12}$ of linear transformations on $V  \otimes V \otimes V$, where ~$R: V  \otimes V\to V  \otimes V$ is a linear operator and $R^{ij}$ means $R$ acting on the $i$th and $j$th components. A \emph{set-theoretical solution} of this equation is a pair~$(X,S)$ such that $X$ is a basis for $V$ and $S : X \times X \rightarrow X \times X$ is a bijective map.
The map $S$ is defined by  $S(x,y)=(g_{x}(y),f_{y}(x))$, where $f_x, g_x:X\to X$ are functions  for all  $x,y \in X$. The pair~$(X,S)$ is \emph{non-degenerate} if for any  $x\in X$, the functions $f_{x}$ and $g_{x}$ are bijections. It is  \emph{involutive} if $S\circ S = Id_X$, and  \emph{braided} if $S^{12}S^{23}S^{12}=S^{23}S^{12}S^{23}$, where the map $S^{ii+1}$ means $S$ acting on the $i$-th and $(i+1)$-th components of $X^3$.
It is said to be \emph{symmetric} if  it is involutive and braided.
Let $\alpha:X \times X \rightarrow X\times X$ be defined by $\alpha(x,y)=(y,x)$, and let $R=\alpha \circ S$, then $R$ satisfies the QYBE if and only if  $(X,S)$ is  braided \cite{etingof}. The solution $(X,S)$ is \emph{the trivial solution} if $g_{x}=f_{x}=Id_X$, $\forall x \in X$.
\begin{defn} \label{def_struct_gp} Let ~$(X,S)$ be a non-degenerate symmetric set-theoretical solution. The \emph{structure group} of $(X,S)$ is presented by   $\operatorname{Gp} \langle X\mid\ xy = g_x(y)f_y(x)\ ;\ x,y\in X \rangle\label{equation:structuregroup1}$. 
\end{defn}
For each $x \in X$,  there are unique $y,z \in X$ such that~$S(x,y) = (x,y)$ and~$S(z,x) = (z,x)$, since  $g_x,f_x$ are bijective and $S$ is involutive. So, the presentation of~$G(X,S)$ contains $\frac{n(n-1)}{2}$ defining relations, where $\mid X \mid =n$. Note that the structure group of the trivial solution is the free abelian group generated by $X$.
 \begin{ex} \label{exemple:exesolu_et_gars} Let $X = \{x_1,x_2,x_3,x_4\}$, and ~$S: X\times X\to X\times X$ be defined by $S(x_i,x_j)=(x_{g_{i}(j)},x_{f_{j}(i)})$ where~$g_i$ and $f_j$ are permutations on~$\{1,2,3,4\}$ as follows: $g_{1}=(2,3)$, $g_2=(1,4)$, $g_{3}=(1,2,4,3)$, $g_{4}=(1,3,4,2)$; $f_{1}=(2,4)$, $f_2=(1,3)$, $f_{3}=(1,4,3,2)$, $f_{4}=(1,2,3,4)$.
The solution  $(X,S)$ is  non-degenerate and symmetric and its  structure group $G(X,S)$ has the following defining relations:  $$\begin{array}{ccccc}
 x_{1}x_{2}=x^{2}_{3};& x_{1}x_{3}=x_{2}x_{4};&
 x_{2}x_{1}=x^{2}_{4};\\ x_{2}x_{3}=x_{3}x_{1};&
 x_{1}x_{4}=x_{4}x_{2};&x_{3}x_{2}=x_{4}x_{1}\\
  \end{array}$$
 There are four trivial relations: $x^{2}_{1}=x^{2}_{1},\,x^{2}_{2}=x^{2}_{2},\,x_{3}x_{4}=x_{3}x_{4},\,x_{4}x_{3}=x_{4}x_{3}$.
  \end{ex}
 \begin{defn}\label{def_decomposale}\cite{etingof}
  Let $(X,S)$ be  a non-degenerate symmetric solution.\\
 $(i)$ A subset $Y$ of $X$ is an \emph{invariant} subset if $S(Y \times Y)\subseteq Y \times Y$.\\
 $(ii)$ An invariant subset $Y$  is \emph{non-degenerate} if $(Y,S\mid_{Y\times Y})$ is  non-degenerate and symmetric. \\
 $(iii)$ The solution $(X,S)$ is  \emph{decomposable} if $X$ is the union of two non-empty disjoint non-degenerate invariant subsets. Otherwise, $(X,S)$ is \emph{indecomposable}.
 \end{defn} 
 Two solutions
  $(X,S)$ and $(X',S')$ are \emph{equivalent solutions} if there exists a bijection $\alpha: X \rightarrow X'$ such that $S' \circ \alpha \times \alpha=\alpha \times \alpha \circ S$. There exists a unique indecomposable  solution $(X,S)$ (up to equivalence) for  $\mid X \mid =p$, $p$ a prime. Indeed, in  this case,    $(X,S)$ is a permutation solution, that is all the functions $f_i$ are equal to a  $p$-cycle $f$ and  all the functions $g_i$ are equal to $f^{-1}$. A classification of non-degenerate and symmetric solutions with $X$ up to $8$ elements, considering their decomposability and other properties is given in \cite{etingof}.

P.Etingof, T.Schedler and A.Soloviev show that non-degenerate symmetric solutions, up to equivalence, are in one-to-one correspondence with quadruples $(G,X,\rho, \pi)$, where $G$ is a group, $X$ is a set, $\rho$ is a left action of $G$ on $X$, and $\pi$ is a bijective $1$-cocycle of $G$ with coefficients in $\mathbb{Z}^{X}$, where $\mathbb{Z}^{X}$ is the free abelian group spanned by $X$. Indeed, they show that the group $G(X,S)$ is naturally a subgroup of $\operatorname{Sym(X)} \ltimes \mathbb{Z}^{X}$, such that the $1$-cocycle defined by the projection $G(X,S) \rightarrow \mathbb{Z}^{X}$ is bijective. More precisely they define and prove the following facts:
\begin{thm}\cite{etingof}\label{prop_etingof}
 Let $(X,S)$ be  a non-degenerate symmetric  solution of the quantum Yang-Baxter equation and $G(X,S)$ be its structure group. Let $\operatorname{Sym(X)} $ be the group of permutations of $X$ and $\mathbb{Z}^{X}$ be the free abelian group spanned by $X$. Let the map $\phi_f: G(X,S) \rightarrow  \operatorname{Sym(X)} \ltimes \mathbb{Z}^{X}$ be defined by $\phi_f(x)= f_{x}^{-1}\,t_x$, where $x \in X$ and $t_x$ is  the generator of $\mathbb{Z}^{X}$ corresponding to $x$. Then\\
 $(i)$ The assignment $x \mapsto f_x^{-1}$ is a left action of $G(X,S)$ on $X$.\\
  $(ii)$ Let $a \in G(X,S)$ and $w=t_1^{m_1}t_2^{m_2}..t_n^{m_n} \in \mathbb{Z}^{X}$. Assume  $a$ acts on $X$ via the permutation $f$. Then $a$ acts on $\mathbb{Z}^{X}$ in the following way: $a \bullet t_x= t_{f(x)}$ and  $a \bullet w= t_{f(1)}^{m_1}t_{f(2)}^{m_2}..t_{f(n)}^{m_n}$,  where $\bullet$ denotes the extension of the left action  of $G(X,S)$ on $X$ defined in $(i)$ to $\mathbb{Z}^{X}$.\\
 $(iii)$ The map  $\phi_f$  is a monomorphism. \\
 $(iv)$ The map $\pi: G(X,S) \rightarrow \mathbb{Z}^{X}$, defined by $\pi(g)= w$ if $\phi_f(g)= \alpha\, w$, with $\alpha \in \operatorname{Sym(X)}$ and  $w \in \mathbb{Z}^{X}$,  is a bijective $1$-cocycle satisfying $\pi(a_1a_2)=(a_2^{-1}\bullet \pi(a_1))\,\pi(a_2)$.
\end{thm}
The product in $\operatorname{Sym(X)} \ltimes \mathbb{Z}^{X}$ is defined by: $f_x^{-1}t_x \, f_y^{-1}t_y= f_x^{-1} f_y^{-1}t_{f_y(x)}t_y$ \cite{etingof}.
\subsection{Computation rules for the  $1$-cocycle and its inverse}\label{sec_prelimin1}
We assume  $(X,S)$ is a non-degenerate symmetric set-theoretical solution with $X$ finite of cardinality $n$ and structure group $G(X,S)$. We denote by $\mathbb{Z}^{n}$ the free abelian group generated  by $X$ written multiplicatively. Here, we present some technical preliminary results, in particular we give some computation rules for the  $1$-cocycle $\pi$ and its inverse $\pi^{-1}$ that are implicit in \cite{etingof} and can be  easily derived from \cite{deh_coxeterlike}. The proofs of the lemmas of this section appear in the appendix. With no loss of generality, we assume $\pi(1)=1$.
\begin{lem}\cite{etingof}\label{lem_etingof_inverseT}
Let $T: X \rightarrow X$ be the map defined by $T(x) =f^{-1}_x(x)$. Then the  map $T$ is invertible and $T^{-1}(y)= g^{-1}_y(y)$. Furthermore, $f_x^{-1}\circ T =T \circ g_x$, $\forall x \in X$.
\end{lem}
Note that inductively  $T^{m}(x)=f^{-1}_{T^{m-1}(x)}T^{m-1}(x)=f^{-1}_{T^{m-1}(x)}f^{-1}_{T^{m-2}(x)}..f^{-1}_{T(x)}f^{-1}_{x}(x)$.
\begin{lem}\label{lem_calcul_pi-1}
Let $\pi: G(X,S) \rightarrow \mathbb{Z}^{n}$ be the bijective $1$-cocycle defined in Prop. \ref{prop_etingof}$(iv)$. Let $x_i,x_j  \in X$. Let  $t_i$, $t_j$ be   generators of $\mathbb{Z}^{n}$ corresponding to $x_i$ and $x_j$ respectively. Let $u,v \in \mathbb{Z}^{n}$. Then \\
$(i)$ $\pi(x_i^{-1})=t^{-1}_{f^{-1}_i(i)}$\\
$(ii)$ $\pi^{-1}(t_j)=x_j$.\\
$(iii)$  $\pi^{-1}(t_it_j)=x_{f^{-1}_j(i)}x_j$.\\
$(iv)$ $\pi^{-1}(uv)=\pi^{-1}(\pi^{-1}(v)\bullet u)\,\pi^{-1}(v)$
\end{lem}

\begin{defn}\label{defn_class}
Let $(X, S)$ be  a non-degenerate symmetric solution, with $X$ finite.\\
$(i)$  Let $x,y \in X$. We say $(X,S)$  \emph{satisfies the condition $C$}, if $f_x\circ f_y=Id_X$ and $g_x \circ g_y=Id_X$, whenever $S(x,y)=(x,y)$. The elements $xy$ and $yx$ are called \emph{frozen elements of length $2$}  \cite{chou_godel2}.\\
$(ii)$ We say $(X,S)$ is of \emph{class $m$}, if 
 $f_x\,f_{T(x)}\,f_{T^{2}(x)}\,...\,f_{T^{m-1}(x)}\,=\,Id_X$,  for all $x \in X$.
\end{defn}

\begin{rem}\label{remark_defn_class}
In  Defn. \ref{defn_class}$(ii)$, we use the terminology from \cite{deh_coxeterlike}.  The formulation is not the same as  Defn. 2.2\cite{deh_coxeterlike}, but it is equivalent to it. Note that satisfying condition $(C)$ is equivalent to being of class 2,  since  $f_x \circ f_y=Id_X$ implies $g_x \circ g_y=Id_X$ from Lemma \ref{lem_etingof_inverseT} and  whenever $S(x,y)=(x,y)$, $y=g_x^{-1}(x)=f_x^{-1}(x)$ and the condition $C$ can be rewritten as  $f_xf_{g_x^{-1}(x)}=f_xf_{T(x)}=Id_X$. Note also that being of class $m$ implies $T^m(x)=x$, for all $x \in X$.
\end{rem}

 \begin{ex} In Example \ref{exemple:exesolu_et_gars}, $T(x_1)=x_1,\  T(x_2)=x_2,\ T(x_3)=x_4,\ T(x_4)=x_3$, since we have the four following trivial relations: $x^{2}_{1}=x^{2}_{1},\,x^{2}_{2}=x^{2}_{2},\,x_{3}x_{4}=x_{3}x_{4},\,x_{4}x_{3}=x_{4}x_{3}$. The solution satisfies condition $(C)$, as $f_1^2=f_2^2=f_3f_4=Id_X$.           
   \end{ex}
    
 In \cite{chou_godel2}, the present author and E.Godelle show that  if $(X, S)$ is  a non-degenerate symmetric set-theoretical solution, with $\mid X \mid =n$, that satisfies the condition $C$, then there is a short exact sequence $1 \rightarrow N \rightarrow G(X,S)  \rightarrow W  \rightarrow 1$, where $N$ is a normal free abelian subgroup of $G(X,S)$  and $W$ is a finite group of order $2^n$. Moreover,  $N$ is freely generated by the $n$ frozen elements of $(X,S)$ of length two and $W$  is  a Coxeter-like group, that is a finite quotient that plays the role that Coxeter groups play for the finite-type Artin groups \cite{chou_godel2}. P.Dehornoy extends our result and using another terminology he shows  that the condition $C$ may be relaxed \cite{deh_coxeterlike}. Indeed, he shows that if $(X,S)$ is of class $m$, then for each $x \in X$ there is a chain of trivial relations of the form $xy_1=xy_1,\;y_1y_2=y_1y_2,\;y_2y_3=y_2y_3,..,y_{m-1}x=y_{m-1}x$, $y_i \in X$.   Here, we  call the element $y_1y_2y_3..y_{m-1}x$ the \emph{frozen element of length $m$ ending with $x$}. The subgroup $N$, generated by the $n$ frozen elements of length $m$,   $y_1y_2y_3..y_{m-1}x,\;xy_1y_2y_3..y_{m-1},\;.., y_2y_3..y_{m-1}xy_1$, is  normal, free abelian  of rank $n$ and the group $W$ defined by $G(X,S)/N$ is finite of order $m^n$ and plays the role that Coxeter groups play for finite-type Artin groups \cite{deh_coxeterlike}.

  \begin{ex} In Example \ref{exemple:exesolu_et_gars}, the normal subgroup $N$ is  generated by the four frozen elements of length two: $x_{1}^2,\; x_{2}^2,\;x_{3}x_{4},\;x_{4}x_{3}$; a presentation of the Coxeter-like group $W$ is obtained by adding the following four relations:  $x_{1}^2=1,\; x_{2}^2=1,\;x_{3}x_{4}=1,\;x_{4}x_{3}=1$ to the presentation of $G(X,S)$. 
  \end{ex}

  The following lemma is useful in the characterisation of the frozen 
 element of length $m$ ending with $x \in X$, where $m$ is the class of the solution.
  \begin{lem}\label{lem_form_frozen}
 Assume $(X,S)$ is of class $m$. Let $x \in X$.\\
  $(i)$  $S(x,T^{-1}(x))=(x,T^{-1}(x))$ and $S(T(x),x)=(T(x),x)$.\\
  $(ii)$  $T^{m-1}(x)\,T^{m-2}(x)\,...\,T(x)\,x$ is the  frozen of length $m$ ending with $x$.
   \end{lem} 
 In the following lemma, we give some computation rules for $\pi^{-1}$ and show the frozen elements act trivially on $\mathbb{Z}^{n}$. It means  that the normal subgroup $N$ is contained in the kernel of the left action of $G(X,S)$ on $X$ with the map  $x \mapsto f_x^{-1}$.
     \begin{lem}\label{lem_frozen_act_trivially}
    Assume $(X,S)$ is of class $m$.   Let $x \in X$. 
     Assume  $\pi(x)=t$. \\
     $(i)$   $\pi^{-1}(t^{-1})\,=\,(T^{m-1}(x))^{-1}\,=\,(T^{-1}(x))^{-1}\,=\,(g^{-1}_x(x))^{-1}$\\
     $(ii)$ For $\ell>0$, $\pi^{-1}(t^{\ell})=T^{\ell-1}(x)\,T^{\ell-2}(x)\,...\,T(x)\,x$.\\
       $(iii)$  $\pi^{-1}(t^{m})\,=\theta$, where $\theta$ is the frozen element of length $m$ ending with $x$.\\
        $(iv)$ $\pi^{-1}(t^{\pm m})\bullet w=w$, for all  $w \in \mathbb{Z}^{n}$.\\
        $(v)$  $\pi^{-1}(t^{-m})\,=\theta^{-1}$.\\
        $(vi)$  $\pi^{-1}(t_{i_1}^{\pm m}..t_{i_k}^{\pm m})\bullet w=w$, for all  $w \in \mathbb{Z}^{n}$.
      \end{lem} 
   
    \section{The connection between Garside groups and the QYBE}\label{Sec_background_garside}
    There exists a strong  connection between set-theoretical solutions and Garside groups \cite{chou_art,chou_godel,ddkgm,gateva}. Indeed, in \cite{chou_art}, we show there is a one-to-one correspondence between non-degenerate symmetric set-theoretical solutions of the quantum Yang-Baxter equation and Garside groups with a particular presentation (see Thm. \ref{thm__strgp_garside}). We first begin with  some basic background on Garside groups in Section \ref{sec_bcgd_gars} and then in Section   \ref{secQYBE_Garside} we present some prior results relating Garside groups and set-theoretical solutions of the quantum Yang-Baxter equation.  
     \subsection{Background on Garside monoids and Garside groups}
     \label{sec_bcgd_gars}
     Here, we recall some basic material on Garside theory, and refer to \cite{ddkgm}, \cite{DiM} for  more details.
     We start with some preliminaries. If $M$ is a monoid generated by a set $X$, and if $g\in M$ is the image of the word~$w$ by the canonical morphism from the free monoid on $X$ onto $M$, then we say  \emph{$w$ represents $g$}. A monoid~$M$ is \emph{cancellative} if for every~${e,f,g,h}$ in $M$, the equality~$efg = ehg$ implies~$f = h$. The element $f$ is a \emph{right divisor}  ({\it resp.} a \emph{left divisor}) of $g$ if there is an element $h$ in $M$ such that $g = hf$  ({\it resp.} $g=fh$).  The element $f$ is a \emph{a least common multiple w.r to right-divisibility (left lcm)}  of $g$ and  $h$ in $M$ if $g$ and $h$ are right divisors of $f$ and
          additionally  if there is an element $f'$ such that $g$ and $h$ are right divisors of $f'$, then
             $f$ is  right divisor of $f'$. We denote it by $f=g \vee_R h$.
            \emph{The complement at left of
             $g$ on $h$} is defined to be an element $c\in M$ such that $g\vee_R h= cg$, whenever $g\vee_R h$ exists. We denote it by
             $c=g \setminus h$. The right lcm and  complement at right are defined similarly and denoted respectively by $\vee_L$ and $\tilde{\setminus}$. A monoid~$M$  is left noetherian ({\it resp.}  \emph{right noetherian}) if every sequence~$(g_n)_{n\in\mathbb{N}}$ of elements of $M$ such that $g_{n+1}$ is a left divisor ({\it resp.} a right divisor) of $g_n$ stabilizes. It is noetherian if it is both left and right noetherian. An element~$\Delta$ is \emph{balanced} if it has the same set of right and left divisors. In this case, we denote by~$\Div{\Delta}$ its set of divisors. If~$M$ is a cancellative and noetherian monoid, then left and right divisibilities are partial orders on~$M$. 
             
             \begin{defn} $(1)$ A \emph{locally Garside monoid} $M$ is a cancellative noetherian monoid such that
               any two elements in $M$ have a lcm for left (and right) divisibility if and only if they have a common multiple for left (and right) divisibility. \\
               $(2)$ A \emph{Garside element} of  $M$ is a balanced element $\Delta$ with $\operatorname{Div}(\Delta)$ generating  $M$. \\
             $(3)$ A monoid $M$ is a \emph{Garside monoid} if $M$ is  locally Garside with a Garside element $\Delta$ satisfying $\operatorname{Div}(\Delta)$ is finite. A group is a \emph{Garside group} if it is the group of fractions of  a Garside monoid.
             \end{defn}
      
     Garside groups have been first introduced in~\cite{DePa}. The seminal examples are the braid groups and more generally the  Artin  groups of finite type.  Recall that an element $g\neq 1$ in a monoid is called an \emph{atom} if the equality $g = fh$ implies $f = 1$ or $h = 1$. From the defining properties of a Garside monoid, if $M$ is a  Garside monoid, then  $M$ is generated by its set of atoms, and every atom divides every Garside element. There is no invertible element, except the trivial one, and any two elements in $M$ have a left ({\it resp. right}) gcd and a left ({\it resp. right}) lcm; in particular, $M$ satisfies the Ore's conditions, so it embeds in its (left and right) group of fractions~\cite{Clifford}. The left and right gcd of two Garside elements are Garside elements and coincide; therefore, by the noetherianity property there exists a unique minimal Garside element for both left and right  divisibilities. This element~$\Delta$ will be called \emph{the} Garside element of the monoid.
    
     \begin{defn}\cite{periodic}\label{defn_periodic}
     Let $G$ be a Garside group with Garside element $\Delta$.\\
     $(i)$ An element $g \in G$ is said to be \emph{periodic}  if $g^q$ is conjugate to $\Delta^p$, where  $p\in \mathbb{Z}$,  $q\in \mathbb{Z}_{\geq 1}$, $p$ and $q$ are prime,
   and $q$ is the smallest positive integer such that   $g^q$ is conjugate to a power of  $\Delta$.\\
   $(ii)$ An element $g \in G$ is said to be \emph{primitive} if it is not a non-trivial power of another element, that is if $g=h^k$ with $k \in \mathbb{Z}$, then $k =\pm1$. 
     \end{defn}
     \begin{thm}\label{thm_algo_periodic}
     Let $G$ be a Garside group. \\
     $(i)$ \cite{sibert}There exists a finite-time algorithm that, given an element $g \in G$ and an integer $k \geq 1$, decides whether there exists $h \in G$ such that $h^k=g$ and then finds such an element $h$ if it exists.\\
     $(ii)$ \cite[Prop. 5.2]{periodic} There exists a finite-time algorithm that makes a list of primitive periodic elements, such that each primitive periodic element of $G$ is conjugate to either exactly one element in the list or its inverse.
     \end{thm}
       \subsection{Description of the Garside structure  specific to the  structure group}
       \label{secQYBE_Garside}
     In the following theorem, we describe the  one-to-one correspondence between non-degenerate symmetric set-theoretical solutions of the quantum Yang-Baxter equation and Garside groups with a particular presentation. 
         
      \begin{thm}\cite[Thm.1]{chou_art}\label{thm__strgp_garside}
      (i) Assume that  $\operatorname{Mon} \langle X\mid \Re \rangle$ is a  Garside monoid such that:
      \begin{enumerate}
      \item[(a)] the cardinality of $\Re$ is $n(n-1)/2$, where $n$ is the cardinality of $X$ and each side of a relation in $\Re$ has length 2 and
      \item[(b)] if the  word $x_{i}x_{j}$ appears in $\Re$, then it appears only once.
      \end{enumerate}
       Then, there exists a function $S: X \times X \rightarrow X \times X$ such that $(X,S)$ is  a non-degenerate symmetric set-theoretical solution and $\operatorname{Gp} \langle X\mid \Re \rangle$ is its structure group.\\
      (ii) For every non-degenerate symmetric set-theoretical solution~$(X,S)$, the structure group $G(X,S)$ is a Garside group, whose Garside monoid is as above.
      \end{thm}

      Moreover, the Garside structure of the structure groups has interesting additional properties described below.    
       \begin{prop}\cite{chou_art,chou_godel}\label{prop_simple_lcm}
       $(i)$ Let $x,y \in X$, $x \neq y$. Then $x\setminus y=f_{x}^{-1}(y)$.\\
       (ii) The Garside element~$\Delta$ is the lcm of $X$ for both left and right divisibilities.\\
       (iii) Let $s$ belong to $M$. Then, \begin{center}\begin{tabular}{ll}&$s$ belongs to~$\Div{\Delta}$\\  $\iff$&$\exists X_{\ell}\subseteq X$ such that~$s$ is the  lcm w.r to left-divisibility of~$X_{\ell}$\\$\iff$& $\exists X_{r}\subseteq X$ such that~$s$ is the  lcm w.r to right-divisibility of~$X_r$.\end{tabular}\end{center}(iv) If~$s$ belongs to $\Div{\Delta}$ then the subsets $X_{\ell}$ and $X_r$ defined in Point~(ii) are unique and have the same cardinality.\\
       
       \end{prop}
           Note that for  ~$s \in \Div{\Delta}$,  ~$X_{\ell}=X_{r}$ if and only if $s$ is balanced.  In \cite{chou_godel2}, the present author and E.Godelle show that  if $(X, S)$ is  a non-degenerate symmetric set-theoretical solution, with $\mid X \mid =n$, that satisfies the condition $C$, then there is a normal free abelian subgroup $N$ of $G(X,S)$ such that $G(X,S)/N$  is  a finite quotient group  that plays the role that Coxeter groups play for the finite-type Artin groups (Coxeter-like group), that is there is a bijection between the elements in $G(X,S)/N$ and $\Div{\Delta}$ \cite{chou_godel2}.  P.Dehornoy extends our result and  shows  that if $(X,S)$ is of class $m$, then there is a bijection between the elements in $G(X,S)/N$ and $\Div{\Delta^{m-1}}$ \cite{deh_coxeterlike}.
       
    M.Picantin defines  the notion of  a $\Delta$-pure Garside group. A Garside monoid $M$ is  $\Delta$\verb`-`\emph{pure} if for every $x,y$ in $X$, it holds that $\Delta_{x} = \Delta_{y}$, where $\Delta_{x} = \vee_L \{b \setminus x ; b \in M\}$. He shows if  $G$ is  the group of fraction of a $\Delta$-pure Garside monoid with exponent $e$, then the centre of $G$ is  infinite cyclic generated by $\Delta^{e}$, where $e$ is the order of the conjugation automorphism by $\Delta$  \cite{picantin}. We show the structure group of $(X,S)$ is   $\Delta$-pure Garside  if and only if $(X,S)$ is indecomposable  \cite{chou_art}. So, if $(X,S)$ is indecomposable,   $Z(G)=\langle \Delta^{e} \rangle$ and if it is decomposable, $Z(G)$ is the direct product of cyclic groups. 
      \begin{ex}\label{ex_center_G_3}
              The solution $(X,S)$ given in Ex.\ref{exemple:exesolu_et_gars} is an indecomposable solution. So, the monoid with the same presentation as $G(X,S)$ is a  $\Delta$-pure Garside  monoid and $Z(G(X,S))=\langle \Delta \rangle$, where $\Delta=(x_1x_3)^2=(x_2x_4)^2=(x_3x_2)^2=(x_4x_1)^2$ is the left and right lcm of $X$. For $G_2$ and $G_3$, the structure groups of the (unique) indecomposable solutions for $\mid X \mid =2$ and $\mid X \mid =3$ respectively, the centre is infinite cyclic generated by the Garside element $\Delta$. We have   $Z(G_2)=\langle \Delta_2 \rangle$, where $\Delta_2=x_1^2=x_2^2$ and  $Z(G_3)=\langle \Delta_3 \rangle$, where $\Delta_3=x_1^3=x_2^3=x_3^3$. Using that, we recover the result  $\operatorname{Inn}(G_2)=\mathbb{Z}_2 * \mathbb{Z}_2$ \cite{karass_piet}.
              \end{ex}
     \section{Construction of a group of automorphisms of the structure group}
 \label{sec_mainResult}
\subsection{Description of the construction of the group of automorphisms $\Phi_\pi$}\label{subsec_subg_aut}
 
  Let $(X,S)$ be  a non-degenerate symmetric  solution of the quantum Yang-Baxter equation and $G(X,S)$ be its structure group.  Assume  $\mid X \mid=n$.  Let $\pi: G(X,S) \rightarrow \mathbb{Z}^{n}$, be the bijective $1$-cocycle defined in Prop. \ref{prop_etingof}. Let $\sigma \in GL_n (\mathbb{Z})=\operatorname{Aut}(\mathbb{Z}^{n})$. Depending on the context, $\sigma$ denotes a map in  $\operatorname{Aut}(\mathbb{Z}^{n})$ or a matrix in $GL_n (\mathbb{Z})$.  We define a  map $\varphi: G(X,S) \rightarrow G(X,S)$ in the the following way:
\begin{equation} \varphi: G(X,S) \rightarrow G(X,S), \;\;\;\varphi(a)= \pi^{-1} \circ \sigma \circ \pi(a) \end{equation}
 That is, $\varphi$ is defined such that the following diagram is commutative:\\
\begin{center} \begin{tikzpicture}\label{diagram}
   \matrix (m) [matrix of math nodes,row sep=3em,column sep=4em,minimum width=2em]
   {
      G(X,S) & \mathbb{Z}^{n} \\
      G(X,S) & \mathbb{Z}^{n} \\};
   \path[-stealth]
     (m-1-1) edge node [left] {$\varphi$} (m-2-1)
             edge [double] node [above] {$\pi$} (m-1-2)
     (m-2-1.east|-m-2-2) edge [double]
             node [above] {$\pi$} (m-2-2)
     (m-1-2) edge node [right] {$\sigma$} (m-2-2)
             ;
 \end{tikzpicture} \end{center}
     As $\varphi$ is a bijection of $G(X,S)$,  the question whether $\varphi$ is an automorphism of $G(X,S)$ reduces to the question whether  $\varphi$ is a homomorphism of $G(X,S)$. More generally,  an epimorphism of $G(X,S)$ is necessarily an automorphism of $G(X,S)$, as  $G(X,S)$ is Hopfian. But,  a monomorphism of $G(X,S)$ may not be an automorphism of $G(X,S)$, as  $G(X,S)$ is co-Hopfian.  We give the proof for completion.
 \begin{prop}\label{prop_QYBE_residually_finite}  The group $G(X,S)$ is residually finite, so $G(X,S)$ is Hopfian. Furthermore,   $G(X,S)$ is co-Hopfian.
 \end{prop}
 \begin{proof}
The structure group  $G(X,S)=\, \mathbb{Z}^{n} \rtimes H$, where $\mathbb{Z}^{n}$ is   the free abelian group of rank $n$ and $H$ is a subgroup of $S_n$ the symmetric group of $n$ elements \cite{etingof}. Both $\mathbb{Z}^{n}$ and $H$ are residually finite, so $G(X,S)$ is residually finite \cite{johnson}. This implies further that $G(X,S)$ is Hopfian \cite{malcev,baumslag}. Let $\sigma: \mathbb{Z}^{n} \rightarrow \mathbb{Z}^{n}$, such that $\operatorname{det}(\sigma) \neq 0,\pm 1$. Then $\sigma$ is a monomorphism that is not surjective. Let $\varphi:G(X,S) \rightarrow G(X,S)$ be defined by  $\varphi= \pi^{-1} \circ \sigma \circ \pi$. Assume $\varphi$ is a homomorphism of $G(X,S)$. So,  $\varphi$ is a monomorphism that is not necessarily surjective. That is, $G(X,S)$ is co-Hopfian.
 \end{proof}
 
 We recall some notation from Section \ref{sec_QYBE}. The function $S$ is defined by $S(x_i,x_j)=(x_{g_{i}(j)},x_{f_{j}(i)})$, where $f_i,g_i:X\to X$ are bijections. 
 The bijective $1$-cocycle $\pi: G(X,S) \rightarrow \mathbb{Z}^{n}$ is defined by  $\pi(x_i)= t_i$, where $x_i \in X$, $t_i$ is  the generator of $\mathbb{Z}^{n}$ corresponding to $x_i$ and $\pi$ satisfies $\pi(a_1a_2)=(a_2^{-1}\bullet \pi(a_1))\pi(a_2)$, for $a_1,a_2 \in G(X,S)$, where  $\bullet$ is the extension of the left action $x \mapsto f_x^{-1}$  of $G(X,S)$ to $\mathbb{Z}^{n}$. Given $\sigma \in GL_n (\mathbb{Z})$, we denote by $(w_i)$ the $i$-th column of $\sigma$ and by $w_i$ the element in $\mathbb{Z}^{n}$ representing $(w_i)$ written multiplicatively, that is \begin{equation*}
 w_i=\sigma(t_i)= t_1^{\sigma_{1,i}}..t_n^{\sigma_{n,i}} =\prod_{j=1}^{j=n} t_j^{\sigma_{j,i}}
 \end{equation*}
 \begin{ex}\label{exemple_sigma_not_working}
 We consider Ex. \ref{exemple:exesolu_et_gars}. Let $\sigma \in GL_n (\mathbb{Z})$ be:
  \[ \left( \begin{array}{cccc}
- 4 & -3 & -4 &0 \\
 1 & 1 & 2 & 0\\
  2 & 1 & 1 & 0\\
 0 & 0 & 0 &  -1\end{array} \right)\] 
 We compute the function  $\varphi=\pi^{-1} \circ \sigma \circ \pi$, using Lemma \ref{lem_calcul_pi-1} and Lemma \ref{lem_frozen_act_trivially}: 
 $\varphi(x_1)=\pi^{-1}(w_1)=\pi^{-1}(t_2t_3^{2}t_1^{-4})=x_2(x_4x_3)(x_1^{-4})$,  $\varphi(x_2)=\pi^{-1}(w_2)=\pi^{-1}(t_3t_2t_1^{-1}t_1^{-2})=x_2x_4x_{g_1^{-1}(1)}^{-1}(x_1^{-2})=x_2x_4(x_1^{-3})$,  $\varphi(x_3)=\pi^{-1}(w_3)=\pi^{-1}(t_3t_2^{2}t_1^{-4})=x_3(x_2^{2})(x_1^{-4})$ and  $\varphi(x_4)=\pi^{-1}(t_4^{-1})=x_{g^{-1}_4(4)}^{-1}=x_3^{-1}$. From Prop. \ref{prop_etingof} and Lemma \ref{lem_calcul_pi-1},  the element  $\varphi(x_1)=x_2(x_4x_3)(x_1^{-4})$  acts on the left on $\mathbb{Z}^{n}$ via the permutation  $f_2^{-1}f_4^{-1}f_3^{-1}f_1^{4}=\,f_{2}^{-1}$. The permutations corresponding to $\pi^{-1}(w_2)$, $\pi^{-1}(w_3)$ and $\pi^{-1}(w_4)$ are respectively $f_3$, $f_3^{-1}$ and $f_3$.
 \end{ex}
  \begin{prop}\label{prop_phi_hom'}
  Let  $\sigma \in GL_n (\mathbb{Z})$ and let $\varphi$ be the bijection defined by $\sigma$. Let $1\leq k\leq n$. Let  $(w_k)$ denote the $k$-th column of $\sigma$ and let  $w_k$ be the element in $\mathbb{Z}^{n}$ representing $(w_k)$. Then   
  \begin{equation}\label{equation_hom_action}
   \varphi \in \operatorname{Aut}(G(X,S)) \;
     	\Leftrightarrow	 \; (\pi^{-1}(w_j))^{-1}\bullet w_i=w_{f_j(i)},\; \forall 1 \leq i, j\leq n
     \end{equation}
 That is, $\varphi$ is a homomorphism of groups if and only if  $\sigma(x_j^{-1}\bullet t_i)=(\varphi(x_j))^{-1}\bullet \sigma(t_i)$.
  \end{prop}
  \begin{proof}
The defining relations in $G(X,S)$ have the form: $x_i\,x_j =x_{g_{i}(j)}\,x_{f_{j}(i))}$, with $x_i,x_j \in X$. So, $\varphi$  is a homomorphism of $G(X,S)$ if and only if $\varphi(x_i\,x_j)=\varphi(x_i)\,\varphi(x_j)$, $\forall x_i,x_j \in X$. The map $\pi: G(X,S) \rightarrow \mathbb{Z}^{n}$ is bijective, so $\varphi(x_i\,x_j)=\varphi(x_i)\,\varphi(x_j)$  in $G(X,S)$  if and only if $\pi(\varphi(x_i\,x_j))=\pi(\varphi(x_i)\,\varphi(x_j))$ in $\mathbb{Z}^{n}$.
From commutativity of the diagram \ref{diagram},  $\pi(\varphi(x_i\,x_j))=\sigma \circ \pi(x_i\,x_j)$. From the 
definition of $\pi$, $\pi(x_i\,x_j)=t_{f_j(i)}\,t_j$, so 
$\pi(\varphi(x_i\,x_j))\,=\,\sigma \circ \pi(x_i\,x_j)\,=\, \sigma(t_{f_j(i)}\,t_j)\,=\, \sigma(t_{f_j(i)})\;\sigma(t_j)\,=\, w_{f_j(i)}\;w_j$. On the other hand,  $\pi(\varphi(x_i)\,\varphi(x_j))\,=\,\pi(\pi^{-1}(w_i)\,\pi^{-1}(w_j))\,=$\\ $(\,(\pi^{-1}(w_j))^{-1} \,\bullet\,\pi(\pi^{-1}(w_i))\,)\;\pi(\pi^{-1}(w_j))\,=\, ((\pi^{-1}(w_j))^{-1} \,\bullet\,w_i)\;w_j$, using first the definition of $\varphi$ and next the property of $\pi$.
So, $\varphi$ is a
homomorphism of $G(X,S)$ if and only if $w_{f_j(i)}\;w_j\,=\,((\pi^{-1}(w_j))^{-1} \,\bullet\,w_i)\;w_j$, 
 for all $1\leq i,j \leq n$, that is if and only if $(\pi^{-1}(w_j))^{-1}\bullet w_i=w_{f_j(i)}$, for all $1\leq i,j \leq n$. We have $(\pi^{-1}(w_j))^{-1}\bullet w_i=(\varphi(x_j))^{-1}\bullet\sigma(t_i)$ and $w_{f_j(i)}=\sigma(t_{f_j(i)})=\sigma(x_j^{-1}\bullet t_i)$, so $\varphi$ is a
 homomorphism of $G(X,S)$ if and only if $\sigma(x_j^{-1}\bullet t_i)=(\varphi(x_j))^{-1}\bullet \sigma(t_i)$.
  \end{proof}
  
  \begin{ex} 
    In Ex.\ref{exemple_sigma_not_working}, we find that the element $(\pi^{-1}(w_1))^{-1}$ acts on the left on $\mathbb{Z}^{n}$ via the permutation $f_2$. So, $(\pi^{-1}(w_1))^{-1}\bullet w_3$ is equal to $t_{f_2(1)}^{-4}t_{f_2(2)}^2 t_{f_2(3)}= t_1t_2^{2}t_3^{-4}$,
   which does not represent any column of $\sigma$. So,  from Prop. \ref{prop_phi_hom'}, $\varphi \notin \operatorname{Aut}(G(X,S))$.
    \end{ex}
   We use the following notation. We denote  by $(f_j)$ the permutation matrix corresponding to $f_j$, by $f'_j$ the  permutation  corresponding to the action of $(\pi^{-1}(w_j))^{-1}$ and  by $A_j$  the corresponding  permutation matrix.  We show that the condition for  $\varphi$ to be a homomorphism of $G(X,S)$ can be written in a very simple form. To simplify the notation in the proof, we define the  permutation matrix as a representing matrix, that is for the permutation $(1,2,3)$, its permutation matrix is  \[ \left( \begin{array}{ccc}
  0& 0 & 1 \\
    1 & 0 & 0\\
     0 & 1 & 0 \end{array} \right)\] 
   \begin{prop} \label{prop_phi_hom_matrices}
    Let  $\sigma \in GL_n (\mathbb{Z})$ and let $\varphi$ be the bijection defined by $\sigma$. Then \begin{equation}\label{equation_matrice}
 \varphi \in \operatorname{Aut}(G(X,S)) \;
   	\Leftrightarrow	 \; A_j\,\sigma\,=\,\sigma\,(f_j^{-1}),\; \forall 1 \leq j\leq n
   \end{equation}
   where $(f_j)$ is  the permutation matrix corresponding to $f_j$ and   $A_j$  is the  permutation matrix corresponding to  the action of $(\pi^{-1}(w_j))^{-1}$. 
  \end{prop}
   \begin{proof}
     From Prop. \ref{prop_phi_hom'},   
    $\varphi \in \operatorname{Aut}(G(X,S))$ if and only if $(\pi^{-1}(w_j))^{-1}\bullet w_i=w_{f_j(i)}$, for all $1\leq i,j \leq n$.  We fix arbitrarily  $1 \leq j\leq n$. We show that  $(\pi^{-1}(w_j))^{-1}\bullet w_i$ and $w_{f_j(i)}$  are elements in $\mathbb{Z}^{n}$ representing  the $i$-th  column in the matrices $A_j\,\sigma$ and  $\sigma\,(f_j^{-1})$ respectively. The  $i$-th  column in   the matrix $\sigma\,(f_j^{-1})$ is  $(w_{f_j(i)})$, as multiplying on the right by $(f_j^{-1})$  permutes the columns of $\sigma$ with respect to $f_j$.  Assume  the $i$-th  column $(w_i)$ is represented by  $w_i=t_1^{m_1}t_2^{m_2}..t_n^{m_n}$ in $\mathbb{Z}^{n}$.  The   $i$-th  column in the matrix $A_j\,\sigma$ is equal to  $A_j\,(w_i)$ which is  represented by the element $t_{f'_j(1)}^{m_1}\,t_{f'_j(2)}^{m_2}..t_{f'_j(n)}^{m_n}$, as $A_j$ permutes the rows of $\sigma$ with respect to the permutation $f'_j$.  From Theorem \ref{prop_etingof}$(ii)$,  $t_{f'_j(1)}^{m_1}\,t_{f'_j(2)}^{m_2}..t_{f'_j(n)}^{m_n}$ is equal to $(\pi^{-1}(w_j))^{-1}\bullet w_i$,  so we have $(\pi^{-1}(w_j))^{-1}\bullet w_i=w_{f_j(i)}$, for all $1 \leq i,j\leq n$, if and only if  the  equality of  matrices  $A_j\,\sigma\,=\,\sigma\,(f_j^{-1})$ holds for all  $1 \leq j\leq n$.
   \end{proof}
 We prove the set of matrices in $GL_n (\mathbb{Z})$ that satisfy Eqn. \ref{equation_hom_action} (and Eqn. \ref{equation_matrice})   is a subgroup of $GL_n (\mathbb{Z})$ and the induced set of automorphisms of $G(X,S)$ is a subgroup of  $\operatorname{Aut}(G(X,S))$ that preserves  the free abelian subgroup  $N$ of  $G(X,S)$ (see Sec.\ref{sec_prelimin1}). 
 
    \begin{thm}\label{theo_im_group_auto}
     Let $(X,S)$ be  a non-degenerate symmetric set-theoretical solution of class $m$, with $\mid X \mid =n$. Let $G(X,S)$ be its structure group.  Let  $\theta_1,..,\theta_n$ be the $n$ frozen elements of length $m$ ending with $x_1,..,x_n$ respectively. Let $N$ be the free abelian subgroup of  $G(X,S)$ generated by $\theta_1,..,\theta_n$.  Let  $\pi: G(X,S) \rightarrow \mathbb{Z}^{n}$ be the  bijective $1$-cocycle defined in Prop. \ref{prop_etingof}.
   We define $\Im_\pi$  to be the following set: 
       
       \begin{eqnarray}
      \Im_\pi=\,\{\sigma \in GL_n (\mathbb{Z}) \mid  (\pi^{-1}(w_j))^{-1}\bullet w_i=w_{f_j(i)},\,  \forall 1 \leq j\leq n \} \label{formule_action}\\
       =\,\{\sigma \in GL_n (\mathbb{Z}) \mid  A_j\,\sigma\,=\,\sigma\,(f_j^{-1}),\,  \forall 1 \leq j\leq n \}\label{formule_matrice}
       \end{eqnarray}
     where $w_j$ is the element in $\mathbb{Z}^{n}$ representing the $j$-th  column of $\sigma$, $(f_j^{-1})$ and  $A_j$ are the permutation matrices corresponding to the action of $x_j$ and  $(\pi^{-1}(w_j))^{-1}$ respectively. We define  $\Phi_{\pi}$ to be the following set: 
                   \begin{equation*}\Phi_{\pi} =\{ \varphi: G(X,S) \rightarrow G(X,S), \;\;\;\varphi= \pi^{-1} \circ \sigma \circ \pi\, \mid \sigma \in \Im_\pi \}  \end{equation*} Then\\
    $(i)$ $\Im_\pi \leq  GL_n (\mathbb{Z})$\\
    $(ii)$  $\Phi_{\pi}\leq \operatorname{Aut}(G(X,S))$. \\
    $(iii)$  $\Phi_{\pi}(N)\subseteq N$.
     \end{thm}
    \begin{proof}
   $(i)$,  $(ii)$ The trivial automorphism belongs to $\Phi_{\pi}$, since   $I^{n\times n} \in \Im_\pi$, as  $A_j=(f_j^{-1})$, for all $1 \leq j\leq n$. Let $\varphi \in \Phi_{\pi}$.  From Prop. \ref{prop_phi_hom_matrices} and Prop. \ref{prop_phi_hom'}, 
    $\varphi \in \operatorname{Aut}(G(X,S))$. So,  $\Phi_{\pi} \subseteq \operatorname{Aut}(G(X,S))$. We have  $\Im_\pi$ is a subgroup of $GL_n (\mathbb{Z})$ and $\Phi_{\pi}$ is a subgroup of $\operatorname{Aut}(G(X,S))$, as both are closed under composition and inverses. Indeed, if $\varphi_1= \pi^{-1} \circ \sigma_1 \circ \pi$ and $\varphi_2= \pi^{-1} \circ \sigma_2 \circ \pi$ belong to $\Phi_{\pi}$,  then $\varphi_1 \circ  \varphi_2 \in \operatorname{Aut}(G(X,S))$. As, $\varphi_1 \circ  \varphi_2= \pi^{-1} \circ \sigma_1 \sigma_2 \circ \pi$, we have  $\sigma_1 \sigma_2 \in \Im_\pi$ and 
     $\varphi_1 \circ  \varphi_2 \in \Phi_{\pi}$  from Prop. \ref{prop_phi_hom'}. In the same way, we prove closure under inverses.\\
     $(iii)$    Let $\varphi \in \Phi_{\pi}$, with $\varphi= \pi^{-1} \circ \sigma \circ \pi$.  Assume $\sigma(t_i)=t_1^{m_1}..t_n^{m_n}$, where $1 \leq i \leq n$.
        We show $\varphi(\theta_i) \in N$.  We have $\varphi(\theta_i) =\pi^{-1} \circ\sigma \circ \pi (\theta_i)=\pi^{-1} \circ\sigma (t_i^m)=\pi^{-1}((t_1^{m_1}..t_n^{m_n})^m)=\pi^{-1}((t_1^m)^{m_1}..(t_n^m)^{m_n})=(\pi^{-1}(t_1^m))^{m_1}..(\pi^{-1}(t_n^m))^{m_n}=\theta_1^{m_1}..\theta_n^{m_n} \in N$, using first Lemma \ref{lem_frozen_act_trivially}$(iii)$, next the definition of $\sigma$ and the commutativity in $\mathbb{Z}^n$ and at last Lemma \ref{lem_calcul_pi-1}$(iv)$ and  Lemma \ref{lem_frozen_act_trivially}$(v),(vi)$.
        \end{proof} Note that if the solution is trivial,  $G(X,S)=\mathbb{Z}^{n}$,  we recover $\Phi_{\pi}=\Im_\pi=GL_n (\mathbb{Z})$. From the proof of Thm. \ref{theo_im_group_auto}, we have that the correspondence $\sigma \rightarrowtail \varphi$ satisfies $\sigma_1 \circ \sigma_2 \rightarrowtail \varphi_1 \circ \varphi_2$ and $\sigma^{-1} \rightarrowtail \varphi^{-1}$. T. Gateva-Ivanova and M. Van den Bergh show  the structure group $G(X,S)$ of a non-degenerate symmetric solution $(X,S)$ with $\mid X \mid =n$ is a Bieberbach group of rank  $n$. Indeed, they show  $G(X,S)$ acts freely on $\mathbb{R}^n$ by isometries with fundamental domain $[0,1]^n$  and the action is defined using the $I$-structure \cite{gateva_van} (see also \cite{jespers,jespers_book}). P. Etingof uses the $1$-cocycle $\pi$ to define a free  action of  $G(X,S)$  on $\mathbb{R}^n$ by isometries to  show the classifying space of $G(X,S)$ is a compact manifold of dimension $n$ \cite{etingof_personal}. The study of the automorphisms of a Bieberbach group has been initiated  in \cite{charlap} and  A. Szczepanski uses the tools developed there to study the outer automorphism group of a Bieberbach  group. In \cite{out_bie}, he studies the question when $\operatorname{Out}(G)$,  the outer automorphism group of a Bieberbach  group $G$,   is finite. He shows that for any given 
        $n \geq 3$, there exists a   Bieberbach  group  of rank $n$ with finite  outer automorphism group and  there exists also a    Bieberbach  group  of rank $n$ with infinite  outer automorphism group (quasi-isometric to $\operatorname{GL_k(\mathbb{Z})}$ with $k \leq n-1$ ). For all $n \geq 3$,  we do not know to which category the structure groups belong,  since, for each $n$, the number of structure groups is less than the number of Bieberbach groups of rank $n$.

       In what follows, we consider computational issues and we give some tools that permit an effective construction of the groups $\Im_\pi$ and  $\Phi_{\pi}$.
        In the following lemma, we give for each column $(w_j)$ a characterisation of the set of  possible permutations for $A_j$. This set is clearly finite and the lemma gives  strong restrictions on it.   
       \begin{lem}\label{lem_fjAj_same_structure}
     Let $\sigma \in GL_n (\mathbb{Z})$. Let  $w_j \in \mathbb{Z}^n$ represent the  $j$-th column of $\sigma$. Let  $f'_j$ be the  permutation  corresponding to the action of $(\pi^{-1}(w_j))^{-1}$  and   $A_j$  the corresponding  permutation matrix. If  $A_j\,\sigma\,=\,\sigma\,(f_j^{-1})$, then \\
     $(i)$ $f'_j$ has the same parity as $f_j$.\\
      $(ii)$ $f'_j$ has the same number of fixed points as $f_j$.\\
      $(iii)$ $f'_j \in \langle f_1,f_2,..,f_n \rangle$.\\
      $(iv)$ Assume $A_i\,\sigma\,=\,\sigma\,(f_i^{-1})$, $1 \leq i \leq n$. Then $A_i=A_j$ if and only if $f_i=f_j$.
     \end{lem}
    \begin{proof}
     $(i)$, $(ii)$  From  $A_j\,\sigma\,=\,\sigma\,(f_j^{-1})$, it results that $A_j$ and $(f_j^{-1})$ are similar matrices, so  $\operatorname{det}(A_j)=\operatorname{det}(f_j^{-1})$ and $\operatorname{trace}(A_j)=\operatorname{trace}(f_j^{-1})$. As   $(f_j)$ is a permutation matrix, it is an orthogonal matrix, so  $\operatorname{det}(f_j^{-1})=\operatorname{det}(f_j)$  and $\operatorname{trace}(f_j^{-1})=\operatorname{trace}(f_j)$. Moreover, $\operatorname{det}(f_j)$ and $\operatorname{trace}(f_j)$ represent respectively the parity and the number of fixed points of the permutation $f_j$. The same holds for $A_j$. So, 
     $f'_j$ has the same parity  and the same number of fixed points as $f_j$.\\
     $(iii)$ By induction on the length of $w_j$. If $w_j=t_k^{\epsilon}$,  $\epsilon=\pm 1$,  then  $\pi^{-1}(t_k^{\epsilon})$ is equal to $x_k$ or $x_{g^{-1}_k(k)}^{-1}$, from lemma \ref{lem_calcul_pi-1}. So, the corresponding permutation  $f'_j$ is $f_k^{-1}$ or $f_{g^{-1}_k(k)}$, that is $f'_j \in \langle f_1,f_2,..,f_n \rangle$. Using Lemma \ref{lem_calcul_pi-1}{} $(iv)$ we have inductively $f'_j\in \langle f_1,f_2,..,f_n \rangle$.      $(iv)$ results from cancellation by $\sigma$.
    \end{proof}
    
 \begin{ex}
 We consider Ex.\ref{exemple:exesolu_et_gars}. Assume  $\sigma \in GL_n (\mathbb{Z})$ has first column $(w_1)\,=\, (-4, 2, 8 , 6)^t$, then  the element $(\pi^{-1}(w_1))^{-1}$ acts trivially on $\mathbb{Z}^{n}$ (as the corresponding permutation is $Id_X$), so $A_1=I^{4\times 4}$. From Lemma \ref{lem_fjAj_same_structure}, $\sigma$  cannot be a candidate, as $f_{1}$ is an odd permutation and $Id_X$ an even one. 
 \end{ex}
   
 The following proposition gives a very efficient tool to compute the permutation corresponding to the action of the element $(\pi^{-1}(w_j))^{-1}$, where $w_j$ represents the $j$-th column on $\sigma$: the computations can be done modulo the class of the solution, and whenever the solution is of class $2$ only the parity of the numbers occurring in $\sigma$ is important. 
  \begin{prop}\label{prop_mod_m}
  Assume $(X,S)$ is of class $m$. Let  $\sigma \in GL_n (\mathbb{Z})$.  Let $(w_j)$ denote the $j$-th  column of $\sigma$, with representing element  $w_j=t_1^{m_1}t_2^{m_2}..t_n^{m_n}$ in $\mathbb{Z}^n$,  $m_1,..,m_n \in \mathbb{Z}$.  Let denote  $w'_j= t_{i_1}^{m'_{i_1}}t_{i_2}^{m'_{i_2}}..t_{i_n}^{m'_{i_n}}$, where $m'_i \equiv m_i\, (mod m)$ and $m'_{i_k} \not\equiv 0\, (mod m)$.
   Then $(\pi^{-1}(w_j))^{-1}\bullet w\,=\,(\pi^{-1}(w'_j))^{-1}\bullet w$, for all  $w\in \mathbb{Z}^{n}$.\\ That is,  the action of $(\pi^{-1}(w_j))^{-1}$ is equivalent to the action of $(\pi^{-1}(w'_j))^{-1}$.
  \end{prop}
 \begin{proof}
 In $\mathbb{Z}^n$, $w_j= t_{i_1}^{m'_1}t_{i_2}^{m'_2}..t_{i_n}^{m'_n}t_1^{k_1m}t_2^{k_2m}..t_n^{k_nm}$, $m'_{i_k} \equiv m_i$, $m'_{i_k} \not\equiv 0$  $(mod m)$. So, 
    $\pi^{-1}(w_j)\bullet w=\,(\pi^{-1}(    \pi^{-1}(t_1^{k_1m}t_2^{k_2m}..t_n^{k_nm}) \bullet  t_1^{m'_1}t_2^{m'_2}..t_n^{m'_n})\,\pi^{-1}( t_1^{k_1m}t_2^{k_2m}..t_n^{k_nm}) )\bullet w=$\\
    $\pi^{-1}(t_1^{m'_1}t_2^{m'_2}..t_n^{m'_n})\,\bullet w$, using  first Lemma \ref{lem_calcul_pi-1}$(iv)$ and then Lemma \ref{lem_frozen_act_trivially}$(vi)$.
 \end{proof}
 We conclude Sec. \ref{subsec_subg_aut} with some examples of groups of automorphisms obtained.
   \begin{ex}\label{exemples_aut_group}
   In Ex. \ref{exemple:exesolu_et_gars}, we find \\
     \begin{eqnarray*}
   \Im_\pi=\{ I^{4 \times 4},-I^{4 \times 4}, \left( \begin{array}{cccc}
    0 & 1 & 0 &0 \\
     1 & 0 & 0 & 0\\
     0 & 0 & 0 & 1\\
     0 & 0 & 1 &  0\end{array} \right),
     \left( \begin{array}{cccc}
         0 & -1 & 0 &0 \\
          -1 & 0 & 0 & 0\\
          0 & 0 & 0 & -1\\
          0 & 0 & -1 &  0\end{array} \right) \}\end{eqnarray*}   
          Note that the automorphism $\varphi$ induced by the matrix $-I \in \Im_\pi$ satisfies $\varphi(x_1)=x_1^{-1}$, $\varphi(x_2)=x_2^{-1}$, $\varphi(x_3)=x_4^{-1}$ and $\varphi(x_4)=x_3^{-1}$, using Lemma  \ref{lem_frozen_act_trivially}$(i)$.\\  
    Let  $G_2= \langle x_2,x_3 \mid  x^{2}_{2}=x^{2}_{3}\rangle$ be the structure group of the unique  non-trivial (indecomposable) solution for $\mid X \mid =2$.             
      For $G_2$, we find \[\Im_\pi=\{ \sigma_1= I^{2 \times 2},\sigma_2=-I^{2 \times 2},
            \sigma_3= \left( \begin{array}{cc}
                0 & 1  \\
                 1 & 0
                 \end{array} \right) ,
                  \sigma_4= \left( \begin{array}{cc}
                            0 & -1  \\
                             -1 & 0
                             \end{array} \right) \}\] 
                
     Accordingly, $\Phi_{\pi} =\operatorname{Out}(G_2)= \mathbb{Z}_2 \times \mathbb{Z}_2$ \cite{karass_piet}.  \\
    Let $G_3= \langle x_1, x_2,x_3 \mid  x_1x_3=x^{2}_{2},x_2x_1=x^{2}_{3},
      x_3 x_2=x^{2}_{1}\rangle$,  the structure group of the unique  indecomposable solution for $\mid X \mid =3$ (of class 3), the group  $\Im_\pi$ is:
       \begin{eqnarray*}
        \Im_\pi=\{ \sigma_1= I^{3 \times 3},
            \sigma_2= \left( \begin{array}{ccc}
                0 & 0 & 1  \\
                 1 & 0 & 0 \\
                 0 & 1 & 0
                 \end{array} \right) ,
                  \sigma_3= \left( \begin{array}{ccc}
                             0 & 1 & 0  \\
                             0 & 0 & 1 \\
                             1 & 0 & 0
                             \end{array} \right) 
                              ,
                    \sigma_4= \left( \begin{array}{ccc}
                           -1 & 0 & 0  \\
                           0 & 0 & -1 \\
                         0 & -1 & 0
                   \end{array} \right)
                            ,\\
        \sigma_5 =\left( \begin{array}{ccc}
                              0 & 0 & -1  \\
                              0 & -1 & 0 \\
                            -1 & 0 & 0
                      \end{array} \right)
             ,
            \sigma_6=  \left( \begin{array}{ccc}
                                    0 & -1 & 0  \\
                                    -1 & 0 & 0 \\
                                     0 & 0 & -1
                            \end{array} \right)                                                                 \}
                            \end{eqnarray*}

     And accordingly, $\Phi_{\pi} = \mathbb{Z}_3 \rtimes \mathbb{Z}_2$.  In  section \ref{Sec_indecomposable}, we show  $\Phi_{\pi} \leq  \operatorname{Out}(G(X,S))$.
     We present a decomposable  solution $(X,S)$ with 
            $X = \{x_1,x_2,x_3\}$,  $f_i=g_{i}=(2,3)$ $1 \leq i \leq3$, and  structure group $\langle x_1,x_2,x_3 \mid x_{1}x_{3}=x_{2}x_{1},x_{1}x_{2}=x_{3}x_{1}, x^{2}_{2}=x^{2}_{3}\rangle \simeq \mathbb{Z}\ltimes G_2$. Then $\Im_\pi=\{
                   \left( \begin{array}{ccc}
                       a & b & b  \\
                        c & d & e \\
                        c &  e & d
                       \end{array} \right) :e=d \pm 1,  a(d+e)-2bc=\pm 1; a,c,d,e \in \mathbb{Z}, b \in 2\mathbb{Z} \}$ is an infinite subgroup of $\operatorname{GL_n(\mathbb{Z})}$. Furthermore,  $\Phi_\pi \cap \operatorname{Inn(G)} \neq Id$, as  the inner automorphism obtained from conjugation by $x_1$
           can be represented by the matrix in $\Im_\pi$ satisfying  $a=e=1$ and $b=c=d=0$.  
\end{ex}

 \subsection{Invariance of  the group $\Phi_{\pi}$ under the change of  bijective 1-cocycle}
Consider the one-to-one correspondence between non-degenerate symmetric solutions  with quadruples (see Sec. \ref{subsec_qybe_Backgd}). If  $(X,S)$ and $(X',S')$ are  different solutions, then they  are in  correspondence with different quadruples $(G(X,S),X,\rho, \pi)$ and  $(G(X',S'),X',\rho', \pi')$, which, using the construction  from Sec. \ref{subsec_subg_aut}, give the groups $\Phi_{\pi}$ and $\Phi_{\pi'}$, respectively.  In this context, a question  that arises naturally is: if $G(X,S) \simeq G(X',S')$ then  what is the relation between  the groups  $\Phi_{\pi}$ and  $\Phi_{\pi'}$, as the bijective 1-cocycles $\pi$ and $\pi'$ may be different.
  We show that under some additional conditions,  $G(X,S) \simeq G(X',S')$ if and only if  $(X,S)$ and $(X',S')$ are equivalent solutions (Lemma \ref{lem_equiv_sol}) and in such a case  
we show $\Phi_{\pi}=\Phi_{\pi'}$, that is $\Phi_{\pi}$ is invariant under change of  bijective 1-cocycle.  
 \begin{lem}\label{lem_equiv_sol}
 Let  $(X,S)$ and $(X',S')$ be  non-degenerate symmetric solutions, with structure groups $G(X,S)$ and $G(X',S')$ respectively. If  $(X,S)$ and $(X',S')$ are equivalent solutions, then $G(X,S) \simeq G(X',S')$. Conversely, if  $G(X,S) \simeq^{\alpha} G(X',S')$ and $\alpha(X)=X'$, then $(X,S)$ and $(X',S')$ are equivalent solutions.
 \end{lem}
 \begin{proof}
 We recall that $G(X,S)$ is generated  by $X$ subject to the defining relations $x\,y=S_1(x,y)\,S_2(x,y)$ for all $x,y \in X$, where $S_1(x,y)$ and $S_2(x,y)$ denote the first and second component of $S(x,y)$. If  $(X,S)$ and $(X',S')$ are equivalent solutions, then 
 $G(X',S')$ is generated  by $X'$ subject to the defining relations $\alpha(x)\,\alpha(y)=\alpha(S_1(x,y))\,\alpha(S_2(x,y))$,  for all $x,y \in X$. So, $G(X,S) \simeq G(X',S')$. Conversely, if  $G(X,S) \simeq^{\alpha} G(X',S')$ and $\alpha(X)=X'$, then $\alpha(x)\,\alpha(y)=\alpha(S_1(x,y))\,\alpha(S_2(x,y))$, for all $x,y \in X$,  are  relations in $G(X',S')$. As, $\alpha(x)\,\alpha(y)=S'_1(\alpha(x)\,,\alpha(y))S'_2(\alpha(x)\,,\alpha(y))$ holds also in $G(X',S')$, we have $(X,S)$ and $(X',S')$ are equivalent solutions.
 \end{proof}
  \begin{prop} \label{prop_equiv_sol}
   Let  $(X,S)$ and $(X',S')$ be  non-degenerate symmetric solutions, with structure groups $G(X,S)$, $G(X',S')$,  bijective $1$-cocycles  $\pi: G(X,S) \rightarrow \mathbb{Z}^{n}$,  $\pi': G(X',S') \rightarrow \mathbb{Z}^{n}$, and groups of automorphisms  $\Phi_{\pi}$, $\Phi_{\pi'}$, respectively. 
   Assume $(X,S)$ and $(X',S')$ are equivalent solutions and  $\pi(x_i)=\pi'(x_i)=t_i$, with $t_i$, $1 \leq i \leq n$,  a generator of $\mathbb{Z}^{n}$. Assume we have one of the following three situations:\\
\begin{figure}[here]
\centering
\begin{tikzpicture}
    \matrix (m) [matrix of math nodes,row sep=3em,column sep=4em,minimum width=2em]
    {
     \mathbb{Z}^{n} & G(X,S)\simeq G(X',S') & \mathbb{Z}^{n} \\
     \mathbb{Z}^{n} & G(X,S)\simeq G(X',S') & \mathbb{Z}^{n} \\ };
    \path[-stealth]
      (m-1-1) edge node [left] {$\sigma$} (m-2-1)
              edge [double] node [above] {$\pi^{-1}$} (m-1-2)
      (m-2-1.east|-m-2-2) edge [double]
              node [above] {$\pi^{-1}$} (m-2-2)
      (m-1-2) edge node [right] {$\varphi$} (m-2-2)
              ;
               \path[-stealth]
                  (m-1-2) edge node [left] {} (m-2-2)
                   edge [double] node [above] {$\pi'$} (m-1-3)
                    (m-2-2.east|-m-2-3) edge [double]
                            node [above] {$\pi'$} (m-2-3)
                    (m-1-3) edge node [right] {$\sigma'$} (m-2-3)
                            ;
  \end{tikzpicture}\caption{}
\end{figure}
\begin{figure}[here]
\centering
\begin{minipage}{0.45\textwidth}
\centering
 \begin{tikzpicture}\label{case2}
    \matrix (m) [matrix of math nodes,row sep=3em,column sep=4em,minimum width=2em]
    {
       G(X,S)\simeq G(X',S') & \mathbb{Z}^{n} \\
      G(X,S)\simeq G(X',S') & \mathbb{Z}^{n} \\ };
    \path[-stealth]
      (m-1-1) edge node [left] {$\varphi$} (m-2-1)
              edge [double] node [above] {$\pi'$} (m-1-2)
      (m-2-1.east|-m-2-2) edge [double]
              node [above] {$\pi$} (m-2-2)
      (m-1-2) edge node [right] {$\sigma$} (m-2-2)
              ;
  \end{tikzpicture} \caption{}
\end{minipage}\hfill
\begin{minipage}{0.45\textwidth}
\centering
 \begin{tikzpicture}\label{case3}
    \matrix (m) [matrix of math nodes,row sep=3em,column sep=4em,minimum width=2em]
    {
      G(X,S)\simeq G(X',S') & \mathbb{Z}^{n} \\
      G(X,S)\simeq G(X',S') & \mathbb{Z}^{n} \\ };
    \path[-stealth]
      (m-1-1) edge node [left] {$\varphi$} (m-2-1)
              edge [double] node [above] {$\pi$} (m-1-2)
      (m-2-1.east|-m-2-2) edge [double]
              node [above] {$\pi'$} (m-2-2)
      (m-1-2) edge node [right] {$\sigma$} (m-2-2)
              ;
  \end{tikzpicture} \caption{}
\end{minipage}
\end{figure}

  Then, in all the cases, the set of automorphisms obtained is  $\Phi_\pi=\Phi_\pi'$.
  \end{prop}
  \begin{proof}
  Assume the case of Fig. 3.1.  Let $\varphi \in \Phi_{\pi}$ defined by  $\sigma \in \Im_\pi$.
 We define the bijection $\sigma':\mathbb{Z}^{n} \rightarrow \mathbb{Z}^{n}$ by $\sigma'=\pi' \circ \varphi \circ \pi'^{-1}$.  Since $\varphi \in \Phi_{\pi}$, $\varphi \in  \operatorname{Aut}(G(X,S))$. So, from Prop.\ref{prop_phi_hom'} applied to $\pi'$, $\sigma' \in  \Im_{\pi'}$. As, $\varphi=\pi'^{-1} \circ \sigma' \circ \pi'$, we have $\varphi \in \Phi_{\pi'}$, that is 
$\Phi_{\pi} \subseteq \Phi_{\pi'}$. The reverse inclusion is obtained in the same way.
  Assume the case of Fig. 3.2 or Fig. 3.3  That is, $\varphi$ is defined by  $\varphi= \pi^{-1} \circ \sigma \circ \pi'$ or $\varphi= \pi'^{-1} \circ \sigma \circ \pi$. In the case of Fig. 3.2, we obtain the group of automorphisms $\Phi_\pi$, as  $\pi'(x_i)=\pi(x_i)=t_i$.
In the  case of Fig. 3.3, we obtain the group of automorphisms $\Phi_\pi'$, as  $\pi(x_i)=\pi'(x_i)=t_i$.
From the above, we have $\Phi_\pi=\Phi_\pi'$.
\end{proof}
  
 \section{A  group of automorphisms that preserves the Garside structure} \label{sec_generalized}
 We show the set of automorphisms of 
   $G(X,S)$ induced by permutation  matrices in $\Im_\pi$
   is a subgroup of $\operatorname{Aut}(G(X,S))$ that preserves the Garside structure of $G(X,S)$, and more generally the subgroup of automorphisms induced by generalized permutation  matrices in $\Im_\pi$ preserves the properties of an extended  Garside structure of $G(X,S)$. We begin with the description of the subgroup  of generalized permutation matrices in  $\Im_\pi$.
 \subsection{Properties of the subgroup of generalized permutation matrices in  $\Im_\pi$}
 We recall that in  a generalized permutation matrix there is exactly one non-zero entry in each row and each column. Unlike a permutation matrix, where the non-zero entry must be 1, for a generalized permutation matrix  in  $GL_n (\mathbb{Z})$  the non-zero entry is $\pm 1$. We use the following notation.  Let $\sigma$ be a generalized permutation matrix  in $\Im_\pi$. Let  $w_j$ be  the element in $\mathbb{Z}^{n}$ representing the  $j$-th  column $(w_j)$, where $w_j=\sigma(t_j)$. If $\sigma_{i_j,j}=  \pm 1$, we define the permutation  $\dot{\sigma}$ by $\dot{\sigma}(j)=i_j$ and  we write $w_j=t_{\dot{\sigma}(j)}^{\epsilon_j}$, where  $\epsilon_j=\pm 1$.
  \begin{lem}\label{lem_elts_same_orbit_permuatation_column}
   Let  $(X,S)$ be a non-degenerate symmetric  solution, with  structure group  $G(X,S)$. Let  $\Im_\pi$  be the subgroup of $GL_n (\mathbb{Z})$ defined in Theorem \ref{theo_im_group_auto}. Let  $\sigma \in\Im_\pi$. Let  $w_j$ and $w_k$ be  elements in $\mathbb{Z}^{n}$ representing the  $j$-th  column $(w_j)$, and the  $k$-th  column $(w_k)$ of $\sigma$. Assume $x_j,x_k \in X$ belong to the same indecomposable component of  $(X,S)$. Then \\
    $(i)$ $(w_j)$ is a  permutation of the rows of   $(w_k)$.\\
    $(ii)$ If additionally, $\sigma$ is a generalized permutation matrix with  $w_j=t_{i_j}^{\epsilon_j}$, $w_k=t_{i_k}^{\epsilon_k}$ and $\epsilon_j,\epsilon_k=\pm 1$, then $\epsilon_j=\epsilon_k$.
     \end{lem}
  \begin{proof}
  $(i)$ Since $x_j,x_k \in X$ belong to the same indecomposable component of  $(X,S)$, there  exists a permutation $\tilde{f_j} \in \langle f_1,..,f_n \rangle$,  $\tilde{f_j}=f_{j_1}\,f_{j_2}..f_{j_l}$, such that $j=\tilde{f_j} (k)$. The proof is by induction on $l$. If $l=1$, then $\tilde{f_j}=f_{j_1}$ and $w_j= w_{f_{j_1}(k)}$. From Eqn. \ref{formule_action}, $(\pi^{-1}(w_{j_1}))^{-1}\bullet w_k=w_{f_{j_1}(k)}$. So,  $(w_j)$ is a permutation of  the rows of   $(w_k)$ according to the action of  $(\pi^{-1}(w_{j_1}))^{-1}$. Assume $l>1$. We have  $((\pi^{-1}(w_{j_1}))^{-1}\,..(\pi^{-1}(w_{j_l}))^{-1}) \bullet w_k=$ $((\pi^{-1}(w_{j_1}))^{-1}\,..(\pi^{-1}(w_{j_{l-1}})^{-1}) \bullet w_{f_{j_l}(k)}\\ =\,w_{f_{j_1}\,f_{j_2}..f_{j_l}(k)}=w_j$, using the action of $G(X,S)$ and $l$ times Eqn. \ref{formule_action}. So,  $(w_j)$ is a permutation of  the rows of   $(w_k)$ according to the action of   $(\pi^{-1}(w_{j_1}))^{-1}\,..(\pi^{-1}(w_{j_l}))^{-1}$. \\
  $(ii)$  From $(i)$, $(w_j)$ is a permutation of  the rows of   $(w_k)$, so $(w_j)$ and $(w_k)$ have elements with the same sign.
  \end{proof}
    
  We show that if $\sigma$ is a generalized permutation matrix  in  $GL_n (\mathbb{Z})$ , the condition for $\sigma $ to belong to $\Im_\pi$ or  equivalently for $\varphi$  to belong to $\operatorname{Aut}(G(X,S))$  described in  Eqn.\ref{equation_hom_action} has a simple form.
  \begin{prop}\label{prop_phi_aut_permut_class}
   Let  $(X,S)$ be a non-degenerate  symmetric  solution, with  structure group  $G(X,S)$. Let $\sigma$ be a generalized permutation matrix in  $GL_n (\mathbb{Z})$  defined by  $\sigma(t_j)=t_{\dot{\sigma}(j)}^{\epsilon_j}$, where $\epsilon_j=\pm 1$, $1 \leq j \leq n$. Let $\varphi$ be defined by $\sigma$.  Then  \\
   $(i)$ If $\epsilon_j=1$,  $\forall 1 \leq j \leq n$,  then 
    $\varphi \in \operatorname{Aut}(G(X,S))$ $\Leftrightarrow$  $\dot{\sigma} \circ f_j \circ \dot{\sigma}^{-1} = f_{\dot{\sigma}(j)}$, $\forall 1 \leq j \leq n$.\\
   $(ii)$  If  $(X,S)$ is of class 2, then   $\varphi \in \operatorname{Aut}(G(X,S))$ $\Leftrightarrow$  $\dot{\sigma}\circ f_j \circ \dot{\sigma}^{-1} = f_{\dot{\sigma}(j)}$, $\forall 1 \leq j \leq n$.\\
     $(iii)$ If  $(X,S)$ is of class $m>2$, then \\  $\varphi \in \operatorname{Aut}(G(X,S))$ if and only if 
       $ \left\{\begin{array}{lcl}
     \dot{\sigma} \circ f_j \circ \dot{\sigma}^{-1} = f_{\dot{\sigma}(j)}, && \epsilon_j=1\\  \dot{\sigma} \circ f_j \circ \dot{\sigma}^{-1} = f^{-1}_{g^{-1}_{\dot{\sigma}(j)}\dot{\sigma}(j)},&& \epsilon_j=-1
     \end{array}\right\}$
  \end{prop}
  \begin{proof}
  From Prop. \ref{prop_phi_hom'},  $\varphi \in \operatorname{Aut}(G(X,S))$ if and only if $\sigma \in \Im_\pi$  if and only if  Eqn. \ref{equation_hom_action} holds, that is  $(\pi^{-1}(w_{j}))^{-1}\bullet w_i=w_{f_{j}(i)}$ for all $1 \leq i,j \leq n$. We  show  Eqn. \ref{equation_hom_action} is equivalent to the condition cited in each case. Generally, we have  $(\pi^{-1}(w_{j}))^{-1}\bullet w_i=w_{f_{j}(i)}$ $\Leftrightarrow$ $(\pi^{-1}(w_{j}))^{-1}\bullet \sigma(t_{i})=\sigma(t_{f_{j}(i)})$ $\Leftrightarrow$ $(\pi^{-1}(w_{j}))^{-1}\bullet t_{\dot{\sigma}(i)}^{\epsilon}=t_{\dot{\sigma} \circ f_{j}(i)}^{\epsilon}$, from the definition of $\sigma$ and from lemma \ref{lem_elts_same_orbit_permuatation_column}$(ii) $, $\epsilon_i=\epsilon_{f_{j}(i)}=\epsilon$. Now, we need to compute $(\pi^{-1}(w_{j}))^{-1}$ in each case. \\
   $(i)$ If $\epsilon_j=1$,  then $w_j=t_{\dot{\sigma}(j)}$ and $(\pi^{-1}(w_{j}))^{-1} =x_{\dot{\sigma}(j)}^{-1}$. So, $(\pi^{-1}(w_{j}))^{-1}\bullet t_{\dot{\sigma}(i)}^{\epsilon}=t_{\dot{\sigma} \circ f_{j}(i)}^{\epsilon}$,
 $\Leftrightarrow$  
   $ t_{f_{\dot{\sigma}(j)}\circ \dot{\sigma}(i)}^{\epsilon}=t_{\dot{\sigma} \circ f_{j}(i)}^{\epsilon}$  $\Leftrightarrow$  
      $ f_{\dot{\sigma}(j)}\circ \dot{\sigma}(i)=\dot{\sigma} \circ f_{j}(i)$.\\
  $(ii),(iii)$ If $\epsilon_j=-1$,  then $w_j=t_{\dot{\sigma}(j)}^{-1}$ and $\pi^{-1}(w_{j}) =x^{-1}_{g^{-1}_{\dot{\sigma}(j)}\dot{\sigma}(j)}$, from lemma \ref{lem_frozen_act_trivially}$(i)$. So,  $(\pi^{-1}(w_{j}))^{-1} =x_{g^{-1}_{\dot{\sigma}(j)}\dot{\sigma}(j)}$ and $(\pi^{-1}(w_{j}))^{-1}\bullet t_{\dot{\sigma}(i)}^{\epsilon}=t_{\dot{\sigma} \circ f_{j}(i)}^{\epsilon}$,
       $\Leftrightarrow$  
         $ f^{-1}_{g^{-1}_{\dot{\sigma}(j)}\dot{\sigma}(j)}\circ \dot{\sigma}(i)=\dot{\sigma} \circ f_{j}(i)$  $\Leftrightarrow$  
           $\dot{\sigma} \circ f_j \circ \dot{\sigma}^{-1} = f^{-1}_{g^{-1}_{\dot{\sigma}(j)}\dot{\sigma}(j)}$. If  $(X,S)$ is of class 2, then  $f^{-1}_{g^{-1}_{\dot{\sigma}(j)}\dot{\sigma}(j)}=f_{\dot{\sigma}(j)}$ (see remark \ref{remark_defn_class}) and   we have           
           $\sigma \in \Im_\pi$ $\Leftrightarrow$  $\dot{\sigma} \circ f_j \circ \dot{\sigma}^{-1} = f_{\dot{\sigma}(j)}$, $\forall 1 \leq j \leq n$.
  \end{proof}
\begin{ex}
Consider in  Ex. \ref{exemples_aut_group} the group $G_3$ of class $3$ with $f_1=f_2=f_3=(1,2,3)$ and the matrix $\sigma_4 \in \Im_\pi$, with  $\dot{\sigma}_4=(2,3$). For $1 \leq i \leq  3$, $\dot{\sigma}_4 \circ f_i\circ \dot{\sigma}_4^{-1}=\dot{\sigma}_4 \circ (1,2,3) \circ \dot{\sigma}_4^{-1}=(1,3,2)=f_i^{-1}$, so $\sigma_4$ satisfies Prop. \ref{prop_phi_aut_permut_class}$(iii)$, as expected. 
\end{ex}
 \subsection{Preservation of the Garside structure by groups of automorphisms} Our first purpose is to show Thm. \ref{thm_permut_preserve_garside} which states that the group of automorphisms  induced by the  permutation matrices in $\Im_\pi$ preserves entirely the Garside structure.
 \begin{thm}\label{thm_permut_preserve_garside}
  Let  $(X,S)$ be a non-degenerate symmetric  solution of class $m$, with  structure group  $G(X,S)$. Let $\dot{\sigma}$ be a  permutation matrix in  $\Im_\pi$ and let  $\varphi \in \operatorname{Aut}(G(X,S))$ be induced  by $\dot{\sigma}$.  Then 
  $\varphi$ is a length-preserving isomorphism of the bounded lattice $\operatorname{Div}(\Delta)$ w.r to right-divisibility and for all $s,t \in \operatorname{Div}(\Delta)$: \\
   $(i)$  $\varphi(s\vee_R t)= \varphi(s)\vee_R \varphi(t)$.\\  $(ii)$ $\varphi(s\wedge_R t)= \varphi(s)\wedge_R \varphi(t)$\\ $(iii)$ $\varphi(\operatorname{Div}(\Delta))=\operatorname{Div}(\Delta)$, with $\varphi(\Delta)=\Delta$ and $\varphi(1)=1$.\\
   $(iv)$  $\varphi(N)=N$, where $N$ is the free abelian subgroup of rank $n$ of  $G(X,S)$ generated by the frozen elements.\\
   Furthermore,   $\varphi$ induces an automorphism of the finite quotient group  $G(X,S)/N$ and  a permutation of the elements in $\Div{\Delta}$.
   \end{thm}
 In order to prove Thm. \ref{thm_permut_preserve_garside}, we show in the following proposition, that  a bijective function $\varphi$ induced by  a  permutation matrix $\dot{\sigma}$ is an automorphism of   $G(X,S)$ if and only if $\dot{\sigma}$  preserves the left complement in some sense. We use the following notation: if $x_j \setminus x_k=x_m$ and $\dot{\sigma}(m)=m'$, then we write $\dot{\sigma}(j\setminus k)=m'$
 \begin{prop}\label{prop_condition_aut_for_permutation}
  Let  $(X,S)$ be a non-degenerate symmetric  solution, with  structure group  $G(X,S)$. Let $\sigma$ be a generalized permutation matrix in  $GL_n (\mathbb{Z})$  defined by  $\sigma(t_j)=t_{\dot{\sigma}(j)}^{\epsilon_j}$, where $\epsilon_j=\pm 1$, $1 \leq j \leq n$. Let $\varphi$ be defined by $\sigma$. If $\sigma=\dot{\sigma}$ or $(X,S)$ is of class 2, then  \\
  $(i)$  $\varphi \in \operatorname{Aut}(G(X,S))$ if and only if $\dot{\sigma}(j\setminus k)=\dot{\sigma}(j)\setminus \dot{\sigma}(k)$ for all $1 \leq j,k \leq n$.\\
  $(ii)$ If  $\varphi \in \operatorname{Aut}(G(X,S))$, then $\dot{\sigma}$ belongs to $C_{S_n}(T)$ the centralizer of $T$ in $S_n$.
\end{prop}
\begin{proof}
$(i)$ From  Prop. \ref{prop_phi_aut_permut_class}, under these assumptions, $\varphi \in \operatorname{Aut}(G(X,S))$ $\Leftrightarrow$  $\dot{\sigma} \circ f_j^{-1} \circ \dot{\sigma}^{-1} = f^{-1}_{\dot{\sigma}(j)}$, $\forall 1 \leq j \leq n$. Let $ 1 \leq k \leq n$. We evaluate $f^{-1}_{\dot{\sigma}(j)}$ and $\dot{\sigma} \circ f_j^{-1} \circ \dot{\sigma}^{-1}$  on  $\dot{\sigma}(k)$. We have $f^{-1}_{\dot{\sigma}(j)}\dot{\sigma}(k)=\dot{\sigma}(j)\setminus \dot{\sigma}(k)$ and  $\dot{\sigma} \circ f_j^{-1} \circ \dot{\sigma}^{-1}(\dot{\sigma}(k))\,= \,\dot{\sigma} \circ f_j^{-1}(k)=\dot{\sigma}(j\setminus k)$, from Prop.\ref{prop_simple_lcm}$(i)$. So,  $\varphi \in \operatorname{Aut}(G(X,S))$ if and only if $\dot{\sigma}(j\setminus k)=\dot{\sigma}(j)\setminus \dot{\sigma}(k)$ for all $1 \leq j,k \leq n$.\\
$(ii)$ Assume  $\varphi \in \operatorname{Aut}(G(X,S))$. From Prop. \ref{prop_phi_aut_permut_class}$(i)$,$(ii)$,  
$\dot{\sigma} \circ f_j^{-1} = f_{\dot{\sigma}(j)}^{-1}\,\circ \dot{\sigma}$, $\forall 1 \leq j \leq n$.
In particular, we have
$\dot{\sigma} \circ f_j^{-1}(j) = f_{\dot{\sigma}(j)}^{-1}\,\circ \dot{\sigma}(j)$, that is $\dot{\sigma} (T(j)) = T(\dot{\sigma}(j))$.
\end{proof}
 \emph{ Proof of Theorem \ref{thm_permut_preserve_garside}}\\
  \begin{proof}
  $(i)$,  $(ii)$, $(iii)$ By definition $x_i\vee_R x_j= (x_i\setminus x_j)\,x_i$. So, $\varphi(x_i\vee_R x_j)=\,\varphi(x_i\setminus x_j)\varphi(x_i)=\,(\varphi(x_i)\setminus \varphi(x_j))\,\varphi(x_i)=\varphi(x_i)\vee_R \varphi(x_j)$, using first $\varphi \in \operatorname{Aut}(G(X,S))$, next Prop. \ref{prop_condition_aut_for_permutation} and the fact that $\varphi(x_i)=x_j$  if and only  if $\dot{\sigma}(i)=j$ (as $\dot{\sigma}$ is  a  permutation matrix) and at last the definition of the left lcm. Inductively we obtain $(*)$ $\varphi(x_{i_1}\vee_R..\vee_R x_{i_k} )= \varphi(x_{i_1})\vee_R .. \vee_R\varphi(x_{i_k})$.  From Prop. \ref{prop_simple_lcm},  if $s \in \operatorname{Div}(\Delta)$ then there exist a subset of $X$, $X_r(s)=\{x_{i_1},..,x_{i_k}\}$ such that $s$ is the left lcm of $X_r(s)$. So, from $(*)$,  $\varphi(s) \in \operatorname{Div}(\Delta)$ and has the same length as $s$. In the same way, $t \in \operatorname{Div}(\Delta)$ is the left lcm  of $X_r(t)$.  Clearly, $s \vee_R t$ is the left lcm of $X_r(s)\cup X_r(t)$ and $s\wedge_R t$ is the left gcd of $(X_r(s) \cap X_r(t))\cup\{1\}$. So, using $(*)$, we have    $(ii)$ and  $(iii)$. Furthermore, $\varphi(\Delta)=\Delta$ and $\varphi(1)=1$. \\
   $(iv)$ From Thm. \ref{theo_im_group_auto}, $\varphi(N)\subseteq N$ and using the fact that $\sigma$ is a permutation matrix in the proof, we have $\varphi(\theta_i)=\theta_{\dot{\sigma}(i)}$, where $\theta_i$ is the frozen ending with $x_i$. So, $\varphi(N)=N$. Moreover,   $\varphi$ induces an automorphism  $\hat{\varphi}$ of the finite quotient group $G(X,S)/N$ defined by  $\hat{\varphi}(Na)=N\varphi(a)$.
 \end{proof}
We now  consider the group of  automorphisms of $G(X,S)$ induced by the  generalized permutation matrices in $\Im_\pi$. We  show this group preserves a Garside structure which is not exactly  the original one. Indeed, we define an extended Garside structure that takes into account generators with negative sign. This is done in the following way.
Let $X=\{x_1,..,x_n\}$. We define $X^{-}=\{x_1^{-1},..,x_n^{-1}\}$
 and  $\widehat{X}=X \cup X^{-}$. For $X^{-}$, we define left  complement and left lcm in the following way: If $x_ix_j=x_kx_l$ is a defining relation in $G(X,S)$, we define $x_i^{-1}\vee_R x_k^{-1}=x_j^{-1}x_i^{-1}= x_l^{-1}x_k^{-1}$,  $x_i^{-1}\setminus x_k^{-1}=x_j^{-1}$ and $x_k^{-1}\setminus x_i^{-1}=x_l^{-1}$. In the same way, we  define right complement and lcm. Obviously, complement and lcm are  defined for every pair of elements in $ X^{-}$ and are unique. The element $\Delta^{-1}$ is the left and right lcm of $X^-$, using Prop. \ref{prop_simple_lcm}.
 We now extend the definition of left and right lcm
to  $\widehat{X}$ in the following way. 
 If $x_ix_j=x_kx_l$ is a defining relation in $G(X,S)$, then $x_i^{-1}x_k=x_jx_l^{-1}$ and  $x_k^{-1}x_i=x_lx_j^{-1}$  hold in $G(X,S)$ and we can define the following elements: 
 $x_i^{-1}\vee_L x_j$,  $x_k^{-1}\vee_L x_l$,  $x_l^{-1}\vee_R x_k$ and $x_j^{-1}\vee_R x_i$. Note that we use the same symbols for the complement and lcm in each set of elements. If for every pair of elements in $X$
 and  for every pair of elements in $X^{-}$, left and right  lcm are defined and unique, this is not the case for mixed pairs with an element from $X$ and another one from $X^{-}$.  
\begin{prop}\label{prop_gene.permut_preserve_extended_garside}
 Let  $(X,S)$ be a non-degenerate symmetric  solution, with  structure group  $G(X,S)$. Let $\sigma$ be a generalized  permutation matrix in  $\Im_\pi$ and  let $\varphi \in \Phi_\pi$ be induced  by $\sigma$.  Let  $x_i,x_j \in X$. Then 
 $\varphi$ satisfies:\\
 $(i)$ $\varphi(x_i\vee_R x_j)= \varphi(x_i)\vee_R\varphi(x_j)$, and $\varphi(x_i\setminus x_j)=\varphi(x_i)\setminus \varphi(x_j)$, when defined.\\
 $(ii)$   $\varphi(x_i\vee_L x_j)= \varphi(x_i)\vee_L \varphi(x_j)$ and $\varphi(x_i\tilde{\setminus} x_j)=\varphi(x_i)\tilde{\setminus} \varphi(x_j)$, when defined.\\
 $(iii)$ $\varphi(N)=N$, where $N$ is the free abelian subgroup of rank $n$ of  $G(X,S)$ generated by the frozen elements.\\
   Furthermore,   $\varphi$ induces an automorphism of the finite quotient group  $G(X,S)/N$.
   \end{prop}
 \begin{proof}
$(i)$, $(ii)$ Assume $x_ix_j=x_kx_l$ is a defining relation in $G(X,S)$, with $x_i,x_j,x_k,x_l \in X$,  so $x_i \vee_L x_k= x_ix_j=x_kx_l$. The elements  $x_i$ and $x_l$ belong to the same irreducible component and the same holds for  $x_j$ and $x_k$. We have $\varphi(x_i \vee_L x_k)=\varphi(x_i)\varphi(x_j)=\varphi(x_k)\varphi(x_l)$, since  $\varphi$ belongs to $\operatorname{Aut}(G(X,S))$.  Moreover,  as $\sigma$ is a generalized permutation matrix,  $\varphi(x_i),\varphi(x_j),\varphi(x_k),\varphi(x_l) \in \widehat{X}$, with $\varphi(x_i)$, $\varphi(x_l)$ sharing  the same sign and $\varphi(x_j)$, $\varphi(x_k)$ sharing  the same sign, from Lemma \ref{lem_elts_same_orbit_permuatation_column}$(ii)$.  So, the equation  $\varphi(x_i)\varphi(x_j)=\varphi(x_k)\varphi(x_l)$ is necessarily derived
  from a defining relation in $G(X,S)$ and   from our  definition of $\vee$, $\setminus$, $\vee_R$,  $\setminus_L$ in $\widehat{X}$, we have $(i)$, $(ii)$, $(iii)$, $(iv)$.\\
$(v)$ From Thm. \ref{theo_im_group_auto}, $\varphi(N)\subseteq N$.  From the proof of Thm. \ref{theo_im_group_auto} and Lemma \ref{lem_frozen_act_trivially}$(v)$, as $\sigma$ is a generalized permutation matrix,  we have $\varphi(\theta_i)=\pi^{-1}(t_{\dot{\sigma}(i)}^{-m})=\theta^{-1}_{\dot{\sigma}(i)}$, where $\theta_i$ is the frozen ending with $x_i$. So, $\varphi(N)=N$ and  $\varphi$ induces an automorphism of the finite quotient group  $G(X,S)/N$.
 \end{proof}
  \section{Special case of indecomposable solutions with $\Phi_{\pi} \leq \operatorname{Out}(G(X,S))$}\label{Sec_indecomposable}
  In this section, we consider the special case of  indecomposable solutions. In particular, for indecomposable solutions with $n=2,3$  we show $\Phi_{\pi} =\operatorname{Out}(G(X,S))$ for $\mid X \mid=2$ and 
                         $\Phi_{\pi} \leq \operatorname{Out}(G(X,S))$ for $\mid X \mid=3$.     
          For $\mid X \mid=3$,  we do not know whether the inclusion $\Phi_{\pi} \leq \operatorname{Out}(G(X,S))$ is an equality. Another question that arises naturally is whether it is possible to prove $\Phi_{\pi} \leq \operatorname{Out}(G(X,S))$ for general indecomposable solutions or $\Phi_{\pi} \cap \operatorname{Out}(G(X,S))$ is not trivial.  
   \subsection{Properties of $\Im_\pi$  for general indecomposable solutions}
   We show that some properties are specific to  indecomposable  solutions.
  \begin{prop}\label{prop_indec_PROPERTIES}
  Let  $(X,S)$ be an indecomposable non-degenerate symmetric  solution. Let $G(X,S)$ be its structure group. Let $\sigma \in \Im_\pi$ and $\varphi \in \Phi_{\pi}$ defined by $\sigma$.
  Let  $w_j$ be the element in $\mathbb{Z}^{n}$ representing the $j$-th  column $(w_j)$ of $\sigma$. Assume $w_1=t_1^{m_1}\,t_2^{m_2}..t_n^{m_n}$, with $m_1,..,m_n \in \mathbb{Z}$. Then 
   \\
  $(i)$ Each column of $\sigma$ is a permutation of the rows of the first column.\\
  $(ii)$ $m_1+m_2+..+m_n= \pm 1$\\
  $(iii)$ If additionally,   the sum of each row is $m_1+m_2+..+m_n$, then  $\varphi(\Delta)= \Delta^{\pm 1}$,  for all $\varphi \in \Phi_\pi$.\\
   \end{prop}
 \begin{proof}
  $(i)$  From Lemma \ref{lem_elts_same_orbit_permuatation_column}$(i)$, if $x_j$ and $x_k$ belong to the same indecomposable component of $(X,S)$, then $(w_j)$ is a permutation of the rows of   $(w_k)$. The solution $(X,S)$ is  indecomposable, so each $(w_j)$ is obtained by permuting the rows of   $(w_1)$ according to some  permutation (that depends on $j$). \\
   $(ii)$ From $(i)$, in each column of $\sigma$ the same numbers $m_1,..,m_n$ appear (in a different order). The sum of each column   is  then $m_1+m_2+..+m_n$, so the number   $m_1+m_2+..+m_n$ is  an  integer eigenvalue of the matrix $\sigma$, where $\operatorname{det}(\sigma)=\pm 1$.\\ 
 $(iii)$  From $(ii)$,  $m_1+m_2+..+m_n= \pm 1$. Assume first $m_1+m_2+..+m_n=1$. So, $\sigma(t_1t_2..t_n)=\sigma(t_1)\sigma(t_2)..\sigma(t_n)=w_1w_2..w_n=t_1^{m_1+m_2+..+m_n}..t_n^{m_1+m_2+..+m_n}= t_1t_2..t_n$, using $(i)$ and  the assumption. That means $t_1t_2..t_n$ is a fixed element of $\sigma$. It holds that $\pi(\Delta)=t_1t_2..t_n$. Indeed, from Prop.\ref{prop_simple_lcm}$(i)$,  $\Delta$ is the right and left lcm of $X$ and it has length $n$, so for all $1 \leq i \leq n$, $t_i$ appears in $\pi(\Delta)$ and  $\pi(\Delta)$ has length $n$ also. So, $\pi(\Delta)=t_1t_2..t_n$. We have then $\sigma(t_1t_2..t_n)=\sigma \circ \pi (\Delta)=\pi (\varphi (\Delta))=t_1t_2..t_n$, from the commutativity of the diagram. Since $\pi$ is bijective and 
 $\pi (\Delta)=\pi (\varphi (\Delta))=t_1t_2..t_n$, we have  $\varphi(\Delta)= \Delta$.  Assume next $m_1+m_2+..+m_n=-1$. So, as before,  $\sigma(t_1t_2..t_n)= t_1^{-1}t_2^{-1}..t_n^{-1}$. It holds that $\pi(\Delta^{-1})=t_1^{-1}t_2^{-1}..t_n^{-1}$. As before,  for all $1 \leq i \leq n$, $\pi(x_i^{-1})$ appears in $\pi(\Delta^{-1})$. From lemma \ref{lem_calcul_pi-1},  $\pi(x_i^{-1})=t_{T(i)}^{-1}$ and $T$ is bijective, so for all $1 \leq i \leq n$, $t_i^{-1}$ appears in $\pi(\Delta^{-1})$. As    $\pi(\Delta^{-1})$ has length $n$,  $\pi(\Delta^{-1})=t_1^{-1}t_2^{-1}..t_n^{-1}$. Using the same argument as before, $\varphi (\Delta)=\Delta^{-1}$.\\
  \end{proof}
  It seems the assumption added in Prop. \ref{prop_indec_PROPERTIES} $(iii)$  holds for all indecomposable solutions.  As the following example illustrates,  if the solution is decomposable, $\varphi(\Delta)= \Delta^{\pm 1}$ does not necessarily hold.
  \begin{ex}
 Consider the structure group  $G(X,S)= \mathbb{Z}\ltimes G_2$ from Ex. \ref{exemples_aut_group} and take the following $\sigma \in \Im_\pi$:  $\left( \begin{array}{ccc}
                                       1 & 2 & 2  \\
                                       1 & 3 & 2 \\
                                       1 & 2 & 3
                               \end{array} \right)$.     
     Then the induced automorphism $\varphi$ satisfies $\varphi(x_1)=x_1x_2x_2$, $\varphi(x_2)=x_2(x_3x_2)(x_2x_3)(x_1x_1)=x_2 \theta_1 \theta_2 \theta_3$, $\varphi(x_3)=x_3(x_2x_3)(x_3x_2)(x_1x_1)=x_3 \theta_1 \theta_2 \theta_3$, and $\varphi(x_1x_2^2)=x_2^2\Delta^5$, where $\Delta=x_1x_2^2$.   The image $x_2^2\Delta^5$ of the central element  $\Delta$ is also a  central element, as $Z(G)= \langle x_1\rangle \times \langle x_2^2 \rangle$. In this case, we have $\varphi(N)=N$, as   $\varphi(\theta_1)= \theta_1 \theta_2 \theta_3$, $\varphi(\theta_2)= \theta_1^2 \theta_2^2 \theta_3^3$ and $\varphi(\theta_3)= \theta_1^2 \theta_2^3 \theta_3^2$.                   
  \end{ex}    
   \subsection{For indecomposable solutions with $n=2,3$, we show $\Phi_{\pi} \leq \operatorname{Out}(G(X,S))$}
   We consider the special case of  indecomposable solutions with $n=2,3$ and  first we show  that  $\Im_\pi$  is a group of generalized permutation matrices. This means that for these solutions, the extended  Garside structure is preserved by $\Phi_{\pi}$ (from Prop. \ref{prop_gene.permut_preserve_extended_garside}) and the Garside structure is preserved by its subgroup of permutation matrices (from Thm. \ref{thm_permut_preserve_garside}).
   \begin{prop}\label{prop_indec_generalized_permut}
    Let  $(X,S)$ be an indecomposable non-degenerate symmetric  solution, with $\mid X \mid =n=2,3$. Let $G(X,S)$ be its structure group. Then  $\Im_\pi$  is a group of generalized permutation matrices in  $GL_n (\mathbb{Z})$. 
   \end{prop}
   \begin{proof}The proof is a case-by-case checking, using Prop. \ref{prop_indec_PROPERTIES}$(i)$.
   For $n=2$, the unique indecomposable solution has structure group $G_2=\langle x_1,x_2 \mid x_1^2=x_2^2 \rangle$. The set of matrices satisfying Eqn.\ref{equation_matrice} has the following form:
   \[ \{
  \left( \begin{array}{cc}
                                  a & b  \\
                                  b & a 
             \end{array} \right) \mid a,b \in \mathbb{Z},a+b=\pm1
              \}
             \]
   In order to belong to   $GL_n (\mathbb{Z})$, a matrix in this set satisfies $a^2-b^2=\pm 1$, that is $a=0$ and $b=\pm 1$
or $b=0$ and $a=\pm 1$. This gives the set $\Im_\pi$ of generalized permutation matrices described in Ex.\ref{exemples_aut_group}.\\
For $n=3$, the unique indecomposable solution has structure group $G_3=\langle x_1,x_2,x_3 \mid x_1x_3=x_2^2, x_2x_1=x_3^2,x_3x_2=x_1^2\rangle$. The set of matrices satisfying Eqn.\ref{equation_matrice} is a subset of one of the two following sets of matrices:
   \[ \{
  \left( \begin{array}{ccc}
                                  a & b & c \\
                                  b & c & a \\
                                  c & a & b
             \end{array} \right) \mid a,b,c \in \mathbb{Z}, a+b+c=\pm1 \};\;\;
      \{
        \left( \begin{array}{ccc}
                                           a & b & c \\
                                           c & a & b \\
                                           b & c & a
             \end{array} \right) \mid a,b,c \in \mathbb{Z}, a+b+c=\pm1 \} 
             \]
   In order to belong to   $GL_n (\mathbb{Z})$, a matrix in one of these two sets satisfies $(a+b+c)(a^2-a(b+c)+(b^2+c^2-bc)=\pm 1$. Using very basic tools from number theory, this gives the set $\Im_\pi$ of generalized permutation matrices described in Ex.\ref{exemples_aut_group}. 
   \end{proof}
   For $n=4$, there are five  indecomposable solutions, three of them are of class 4 and the two other of class 2. One of the two indecomposable solutions of class 2 is described in Ex. \ref{exemple:exesolu_et_gars} and the corresponding group $\Im_\pi$ is also  a group of generalized permutation matrices in  $GL_n (\mathbb{Z})$ as described in  Ex.\ref{exemples_aut_group}. We conjecture this is the case for all the five indecomposable solutions. More generally, it would be interesting to understand  the group $\Im_\pi$ for indecomposable solutions.
       
    \begin{thm}\label{thm_indec_outer_aut}
    Let  $(X,S)$ be an indecomposable non-degenerate symmetric  solution, with $\mid X \mid =n=2,3$.  Let $G(X,S)$ be its structure group.  Let $\Phi_{\pi}$ be the group of automorphisms of $G(X,S)$ induced by   the group of generalized permutation matrices $\Im_\pi$ from Prop. \ref{prop_indec_generalized_permut}. Then, for $n=2$,
      $\Phi_{\pi} =\operatorname{Out}(G(X,S))$ and for $n=3$,
            $\Phi_{\pi} \leq \operatorname{Out}(G(X,S))$.
     \end{thm}
     \begin{proof}
   For     $n=2$, the group $\Phi_{\pi} \simeq \mathbb{Z}_2 \times \mathbb{Z}_2$, with\[ \Phi_{\pi}= \{
        \varphi_1= Id,
       \varphi_2: \left\{ \begin{array}{ccc}
                                    x_1 & \mapsto & x_2^{-1} \\
                                    x_2 & \mapsto & x_1^{-1} \\
                      \end{array} \right\},
   \varphi_3: \left\{ \begin{array}{ccc}
                                      x_1 & \mapsto & x_2 \\
                                      x_2 & \mapsto & x_1 \\
                        \end{array} \right\}, 
    \varphi_4: \left\{ \begin{array}{ccc}
                                       x_1 & \mapsto & x_1^{-1} \\
                                       x_2 & \mapsto & x_2^{-1} \\
                         \end{array} \right\} \} \]
  Each automorphism $\varphi_i$  in  $\Phi_{\pi}$ corresponds to a matrix $\sigma_i$ in  $\Im_\pi$ (see Ex.\ref{exemples_aut_group}). As $G_2= \mathbb{Z} *_{2\mathbb{Z}} \mathbb{Z}$, using the work of   A. Karrass, A. Pietrowski and D. Solitar on automorphisms of a free product with an amalgamated subgroup \cite{karass_piet},       
         $\Phi_{\pi} =\operatorname{Out}(G_2)$.\\
         For     $n=3$, the group $\Phi_{\pi} \simeq \mathbb{Z}_3 \rtimes \mathbb{Z}_2$, with
         \begin{eqnarray*}
        \Phi_{\pi}= \{
                 \varphi_1= Id,
                \varphi_2: \left\{ \begin{array}{ccc}
                                             x_1 & \mapsto & x_2 \\
                                             x_2 & \mapsto & x_3 \\
                                             x_3 & \mapsto & x_1 
                               \end{array} \right\},
         \varphi_3: \left\{ \begin{array}{ccc}
                                              x_1 & \mapsto & x_3 \\
                                              x_2 & \mapsto & x_1 \\
                                              x_3 & \mapsto & x_2 
                                        \end{array} \right\},   \\                   
           \varphi_4: \left\{ \begin{array}{ccc}
                                                x_1 & \mapsto & x_2^{-1} \\
                                                x_2 & \mapsto & x_1^{-1} \\
                                                x_3 & \mapsto & x_3^{-1} 
                                  \end{array} \right\},
         \varphi_5: \left\{ \begin{array}{ccc}
                                                        x_1 & \mapsto & x_1^{-1} \\
                                                        x_2 & \mapsto & x_3^{-1} \\
                                                        x_3 & \mapsto & x_2^{-1} 
                                          \end{array} \right\},
 \varphi_6: \left\{ \begin{array}{ccc}
                                                x_1 & \mapsto & x_3^{-1} \\
                                                x_2 & \mapsto & x_2^{-1} \\
                                                x_3 & \mapsto & x_1^{-1} 
                                  \end{array} \right\} \} \end{eqnarray*}
          Each automorphism $\varphi_i$  in  $\Phi_{\pi}$ corresponds to a matrix $\sigma_i$ in  $\Im_\pi$ (see Ex.\ref{exemples_aut_group}). We show  $\Phi_{\pi} \cap \operatorname{Inn}(G_3)=\{Id\}$.
                  Assume  that there is  $\varphi \in \Phi_{\pi} \cap \operatorname{Inn}(G_3)$. That is, there exists an element $a \in G$ such that $\varphi(x_i)=ax_ia^{-1}$, for all $1 \leq i \leq n$. Since each $\varphi \in \Phi_{\pi}$ has finite order and $Z(G_3)=\langle \Delta \rangle$ (see Ex. \ref{ex_center_G_3}), we have that $a$ is a periodic element, that is there exists $p \in \mathbb{Z}$ and $q \in \mathbb{Z}_{\geq 1}$ such that $g^q=\Delta^p$ (see Defn. \ref{defn_periodic}). From Theorem \ref{thm_algo_periodic}, there exists a finite-time algorithm that finds all the primitive periodic elements. Applying the algorithm described in \cite[Prop. 5.2]{periodic} (or in \cite[Prop.2.13]{sibert}), we find the following set of roots of $\Delta$: 
                  $\{x_1,x_2,x_3,x_1^2,x_2^2,x_3^2,\Delta\}$. If $a$ is equal to one of these roots of $\Delta$, then $\varphi$ fixes necessarily one of the generators. As none of the automorphisms in  $\Phi_{\pi}$  except $Id$ fixes a generator, we have  $\varphi=Id$.
     \end{proof}
   
 This result holds also for the indecomposable solution with $n=4$ presented in Ex. \ref{exemple:exesolu_et_gars} using a case by case computation.  We conclude this section with the following observation. If we consider the unique indecomposable solution for $\mid X \mid =5$, which is a permutation solution with $f=(1,2,3,4,5)$ and $g=(1,5,4,3,2)$, then  the group of matrices  $\Im_\pi$  is not necessarily a group of  generalized permutation matrices. As an example,  the circulant matrix  with first column $(1,-1,1,0,0)^t$ belongs to $\Im_\pi$.
 The existence of matrices that are not necessarily  generalized permutation matrices in $\Im_\pi$ illustrates the difficulty in understanding the connection between $\Phi_{\pi}$ and $\operatorname{Out}(G(X,S))$ for general solutions.
  \section{Application of our results to the quotient group $G(X,S)/N$} \label{sec_Application_to_quotientgroup}
 \subsection{Definition of a bijective $1$-cocycle $\tilde{\pi}$  from the quotient group $G(X,S)/N$}
 Let $(X,S)$ be a non-degenerate, symmetric solution with structure group $G(X,S)$. Let $\pi$ be the bijective 1-cocycle defined in section  \ref{sec_prelimin1}. Assume 
   $(X,S)$  is of class $m$. Let $N$ be the free abelian group generated by the $n$ frozen elements of length $m$, $\theta_1,..,\theta_n$, and 
    $W$ be the finite quotient group  $G(X,S)/N$.
 We show that the normal subgroup $N$ acts trivially on $\mathbb{Z}^n$ and therefore there exists a bijective 1-cocycle 
 $\tilde{\pi}: W \rightarrow \mathbb{Z}^n/(\pi(N)) $ that is induced by $\pi$.
 \begin{lem}
 Let $w \in \mathbb{Z}^n$. Then $N \bullet w=w$. Moreover, the map $\tilde{\pi}: W \rightarrow \mathbb{Z}^n/(\pi(N)) $ defined by $\tilde{\pi}(aN)=\pi(a)\,(\pi(N)) $ is a bijective 1-cocycle.
   \end{lem}
 \begin{proof}
 From Lemma \ref{lem_frozen_act_trivially} $(iii),(iv)$, $\theta_i \bullet w =w$ for all $1 \leq i \leq n$. So, $N \bullet w=w$. A direct computation shows  $\tilde{\pi}(aN)=\pi(a)\,(\pi(N)) $ is a bijective 1-cocycle.
 \end{proof}
 
 \begin{rem}
Assuming $m=2$, that is the solution $(X,S)$ satisfies condition C, $W$ is a subgroup of  $S_n \ltimes \operatorname{Diag}(n, \mathbb{Z})$ that is a Weyl (or finite Coxeter) group of type  $C_n$ and we give  another equivalent definition of the bijective 1-cocycle 
 $\tilde{\pi}$  induced by $\pi$.
 Indeed, let $\tau: W \rightarrow S_n \ltimes \operatorname{Diag}(n, \mathbb{Z})$ be defined by $\tau(x_iN)= (f_i^{-1},I^{(i)})$, where $\operatorname{Diag}(n, \mathbb{Z})$ is the subgroup of diagonal matrices in $GL_n(\mathbb{Z})$ and $I^{(i)}$ is given by $I^{(i)}_{i,i}=-1$ and $I^{(i)}_{j,j}=1$, for $j \neq i$.  A direct checking shows  $\tau: W \rightarrow S_n \ltimes \operatorname{Diag}(n, \mathbb{Z})$ is a monomorphim. The bijective 1-cocycle  $\tilde{\pi}$  induced by $\pi$ is given by  $\tilde{\pi}:W \rightarrow \operatorname{Diag}(n, \mathbb{Z})$ and  defined by $\tilde{\pi}(aN)=A$, whenever $\tau(aN)=(f,A)$. 
 \end{rem}
 \subsection{A group of automorphisms of the quotient group $G(X,S)/N$}
  Using the bijective 1-cocycle $\tilde{\pi}: W \rightarrow \mathbb{Z}^n/(\pi(N))$ and the construction of Section  \ref{subsec_subg_aut} applied to $W$, we find a subgroup of automorphisms $\Phi_{\tilde{\pi}}$ of $W$.  A question that arises naturally is  whether $\hat{\varphi}:W \rightarrow W$ belongs to  $\Phi_{\tilde{\pi}}$,  whenever  $\hat{\varphi}$ is an automorphism of $W$ induced by  $\varphi \in \Phi_\pi$.
 \begin{prop}
 Let $\varphi \in \Phi_\pi$. Assume  $\varphi(N)=N$. Let $\hat{\varphi}:W \rightarrow W$
be defined  by $\hat{\varphi}(Na)=N\varphi(a)$. Then  $\hat{\varphi}$ belongs to  $\Phi_{\tilde{\pi}}$. 
 \end{prop}
 \begin{proof}
 Assume $\varphi(N)=N$.  From the definition of $\varphi$, $\pi \circ \varphi=  \sigma \circ \pi$, so  $ \sigma( \pi(N))=\pi(N)$ and $\tilde{\sigma}:\mathbb{Z}^n/\pi(N) \rightarrow \mathbb{Z}^n/\pi(N)$ is well-defined. Furthermore,  $\tilde{\sigma}$ induces  $\hat{\varphi}$ such that the following diagram is  commutative and $\hat{\varphi}$ belongs to  $\Phi_{\tilde{\pi}}$:

  \begin{tikzpicture}
    \matrix (m) [matrix of math nodes, row sep=3em,
      column sep=3em]{
      &  G& &  \mathbb{Z}^n \\
     W &  &\mathbb{Z}^n/\pi(N) & \\
      & G & & \mathbb{Z}^n \\
      W &  & \mathbb{Z}^n/\pi(N) & \\};
    \path[-stealth]
      (m-1-2) edge [densely dotted] node [right] {$\pi$}(m-1-4) 
      edge[double] 
      (m-2-1)
              edge[densely dotted] node [below] {$\varphi$} (m-3-2) 
      (m-1-4) edge  [densely dotted] node [below] {$\sigma$} (m-3-4)
               edge [double] (m-2-3)
      (m-2-1) edge node [right] {$\tilde{\pi}$}  (m-2-3)
              edge (m-2-3) 
              edge node [right] {$\hat{\varphi}$} (m-4-1)
      (m-3-2) edge [densely dotted] node [right] {$\pi$}  (m-3-4)
              edge  [double] (m-4-1)
      (m-4-1) edge node [right] {$\tilde{\pi}$}(m-4-3)
      (m-3-4) edge [double] (m-4-3)
      (m-2-3) edge node [below] {$\tilde{\sigma}$} (m-4-3)
              edge (m-4-3);
  \end{tikzpicture}
 \end{proof}
 \section*{Appendix: Proofs of Lemmas \ref{lem_calcul_pi-1}, \ref{lem_form_frozen}, \ref{lem_frozen_act_trivially}}
 \begin{lem*}\textbf{\ref{lem_calcul_pi-1}}
 Let $\pi: G(X,S) \rightarrow \mathbb{Z}^{n}$ be the bijective $1$-cocycle defined in Prop. \ref{prop_etingof}$(iv)$. Let $x_i,x_j  \in X$. Let  $t_i$, $t_j$ be   generators of $\mathbb{Z}^{n}$ corresponding to $x_i$ and $x_j$ respectively. Let $u,v \in \mathbb{Z}^{n}$. Then \\
 $(i)$ $\pi(x_i^{-1})=t^{-1}_{f^{-1}_i(i)}$\\
 $(ii)$ $\pi^{-1}(t_j)=x_j$.\\
 $(iii)$  $\pi^{-1}(t_it_j)=x_{f^{-1}_j(i)}x_j$.\\
 $(iv)$ $\pi^{-1}(uv)=\pi^{-1}(\pi^{-1}(v)\bullet u)\,\pi^{-1}(v)$
 \end{lem*}
 \begin{proof}
 $(i)$ and $(ii)$ result from the definition of $\pi$ and the assumption that $\pi(1)=1$. \\
 $(iii)$
 We have  $\pi(x_{f^{-1}_j(i)}x_j)=t_{f_j(f^{-1}_j(i))}\,t_j=t_i\,t_j$, from the definition of $\pi$ (Prop. \ref{prop_etingof}). \\
  $(iv)$ We have $\pi(\pi^{-1}(\pi^{-1}(v)\bullet u)\,\pi^{-1}(v))=\,((\pi^{-1}(v))^{-1})\bullet (\pi^{-1}(v)\bullet u)\,v=\,uv$, first using the definition of $\pi$ and then the fact that $\bullet$ is an action on $\mathbb{Z}^{n}$.
 \end{proof}

 \begin{lem*}\textbf{\ref{lem_form_frozen}}
 Assume $(X,S)$ is of class $m$. Let $x \in X$.\\
  $(i)$  $S(x,T^{-1}(x))=(x,T^{-1}(x))$ and $S(T(x),x)=(T(x),x)$.\\
  $(ii)$  $T^{m-1}(x)\,T^{m-2}(x)\,...\,T(x)\,x$ is the  frozen of length $m$ ending with $x$.
   \end{lem*} 
   \begin{proof} 
   $(i)$ We have  $S(T(x),x)\,=\,S(f^{-1}_x(x), x)= (..,f_x(f^{-1}_x(x))=(..,x)=(T(x),x)$, using first the definition of $S$ and then the non-degeneracy and  involutivity of $S$. Using the same arguments, $S(x,T^{-1}(x))\,=\,S(x,g^{-1}_x(x))\,=\,(x,..)=(x,T^{-1}(x))$.
   \\
     $(ii)$  From $(i)$,  $S(T(x),x)=(T(x),x)$, so $S(T^2(x),T(x))=(T^2(x),T(x)),\;..$ \\$S(T^{m-1}(x),T^{m-2}(x))
     \, =\,(T^{m-1}(x),T^{m-2}(x))$ and $S(x,T^{m-1}(x))
         \, =\,(x,T^{m-1}(x))$, since  $T^m(x)=x$. So,  $T^{m-1}(x)\,T^{m-2}(x)\,...\,T(x)\,x$ is the  frozen of length $m$ ending with $x$.
     \end{proof}

      \begin{lem*}\textbf{\ref{lem_frozen_act_trivially}}
     Assume $(X,S)$ is of class $m$.   Let $x \in X$. 
      Assume  $\pi(x)=t$. \\
      $(i)$   $\pi^{-1}(t^{-1})\,=\,(T^{m-1}(x))^{-1}\,=\,(T^{-1}(x))^{-1}\,=\,(g^{-1}_x(x))^{-1}$\\
      $(ii)$ For $\ell>0$, $\pi^{-1}(t^{\ell})=T^{\ell-1}(x)\,T^{\ell-2}(x)\,...\,T(x)\,x$.\\
        $(iii)$  $\pi^{-1}(t^{m})\,=\theta$, where $\theta$ is the frozen element of length $m$ ending with $x$.\\
         $(iv)$ $\pi^{-1}(t^{\pm m})\bullet w=w$, for all  $w \in \mathbb{Z}^{n}$.\\
         $(v)$  $\pi^{-1}(t^{-m})\,=\theta^{-1}$.\\
         $(vi)$  $\pi^{-1}(t_{i_1}^{\pm m}..t_{i_k}^{\pm m})\bullet w=w$, for all  $w \in \mathbb{Z}^{n}$.
       \end{lem*} 
       \begin{proof}
        $(i)$ We have $\pi((T^{m-1}(x))^{-1})=(t_{f^{-1}_{T^{m-1}(x)}T^{m-1}(x)} )^{-1}=(t_{T^{m}(x)})^{-1}=t^{-1}$, using first Lemma \ref{lem_calcul_pi-1}$(i)$, next the definition of $T$ and at last since  the class of the solution is $m$, we have $T^{m}(x)=f^{-1}_{T^{m-1}(x)}f^{-1}_{T^{m-2}(x)}..f^{-1}_{T(x)}f^{-1}_{x}(x)=\,f^{-1}_{T^{m-1}(x)}T^{m-1}(x)=x$. From lemma \ref{lem_etingof_inverseT}, $T^{-1}(x)=\,g^{-1}_x(x)$ and $T^{-1}(x)=T^{m-1}(x)$. \\
        $(ii)$ By induction on $\ell$. For $\ell=2$, $\pi^{-1}(t^2)=f^{-1}_x(x)\,x=T(x)x$, using  Lemma \ref{lem_calcul_pi-1}$(ii)$. For $\ell>2$, 
            $\pi^{-1}(t^{\ell})=\,\pi^{-1}(\pi^{-1}(t^{})\bullet t^{\ell-1})\,\pi^{-1}(t)=\,  \pi^{-1}(x\bullet t^{\ell-1})\,x=\, \pi^{-1}(t_{f_x^{-1}(x)}^{\ell-1})x=$\\$ T^{\ell-2}(T(x))\,..T(T(x))\,T(x)\,x=T^{\ell-1}(x)\,...\,T(x)\,x$, using first Lemma\ref{lem_calcul_pi-1}$(iii),(ii)$ and then the induction assumption for  $t_{f_x^{-1}(x)}$ and the fact that  $\pi^{-1}(t_{f_x^{-1}(x)})=T(x)$.\\   
            $(iii)$ From $(ii)$, if $\pi(x)=t$, then  $\pi^{-1}(t^{m})=T^{m-1}(x)\,T^{m-2}(x)\,...\,T(x)\,x$ and from lemma \ref{lem_form_frozen}$(ii)$,  this is the frozen element of length $m$ ending with $x$. \\
             $(iv)$ The element $\pi^{-1}(t^{m})=T^{m-1}(x)\,T^{m-2}(x)\,...\,T(x)\,x$  acts on $\mathbb{Z}^{n}$ via the permutation $f^{-1}_{T^{m-1}(x)}..f^{-1}_{T(x)}f^{-1}_{x}$.  Since $(X,S)$ is of class $m$, $f^{-1}_{T^{m-1}(x)}..f^{-1}_{T(x)}f^{-1}_{x}=Id_X$, that is 
            $\pi^{-1}(t^{m})\bullet w=w$,  for all  $w \in \mathbb{Z}^{n}$. Let  $w \in \mathbb{Z}^{n}$. We show  $\pi^{-1}(t^{-m})\bullet w=w$. On one hand,  $\pi^{-1}(t^{-m}t^{m})\bullet w=w$. On the other hand,
            $\pi^{-1}(t^{-m}t^{m})=\,\pi^{-1}(\pi^{-1}(t^{m})\bullet t^{-m})\,\pi^{-1}(t^{m})$,  from Lemma \ref{lem_calcul_pi-1}$(iv)$. As $\pi^{-1}(t^{m})$ acts trivially on    $\mathbb{Z}^{n}$, 
             $\pi^{-1}(t^{-m}t^{m})\bullet w=\,\pi^{-1}(t^{-m})\bullet w$. So,  
             $\pi^{-1}(t^{-m})\bullet w=w$.\\
             $(v)$ From $(iv)$,   $\pi^{-1}(t^{-m}t^{m})=\,\pi^{-1}(\pi^{-1}(t^{m})\bullet t^{-m})\,\pi^{-1}(t^{m})= \,\pi^{-1}( t^{-m})\,\pi^{-1}(t^{m})=1$, so $\pi^{-1}(t^{-m})$ is the inverse of $\pi^{-1}(t^{m})$, that is $\pi^{-1}(t^{-m})\,=\theta^{-1}$.\\
            $(vi)$ By induction on $k$. The case  $k=1$ results from $(iii)$. Assume $k>1$. Then from  Lemma \ref{lem_calcul_pi-1}$(iii)$,  $\pi^{-1}(t_{i_1}^{\pm m}t_{i_2}^{\pm m}..t_{i_k}^{\pm m})=$      
            $\pi^{-1} (\pi^{-1}(t_{i_2}^{\pm m}..t_{i_k}^{\pm m})\bullet t_{i_1}^{\pm m})\,\pi^{-1}(t_{i_2}^{\pm m}..t_{i_k}^{\pm m})$.  So,  
            $\pi^{-1}(t_{i_1}^{\pm m}t_{i_2}^{\pm m}..t_{i_k}^{\pm m})\bullet w=$      
                 $\pi^{-1} (\pi^{-1}(t_{i_2}^{\pm m}..t_{i_k}^{\pm m})\bullet t_{i_1}^{\pm m})\,\bullet w=   
             \pi^{-1}(t_{i_1}^{\pm m})\bullet w =w$, using the 
             induction assumption on  
            $\pi^{-1}(t_{i_2}^{\pm m}..t_{i_k}^{\pm m})$ and then the result for $k=1$.
            \end{proof}

\bigskip\bigskip\noindent
{ Fabienne Chouraqui,}

\smallskip\noindent
University of Haifa at Oranim, Israel.
\smallskip\noindent
E-mail: {\tt fchoura@sci.haifa.ac.il}

\end{document}